\documentclass[11pt, a4paper]{amsart}
\pdfoutput=1

\usepackage{amsmath}

\usepackage{amssymb}
\usepackage{stmaryrd}
\usepackage{amsthm}

\usepackage{amscd}
\usepackage[all]{xy}
\usepackage{url}

\newtheorem{Thm}{Theorem}[subsection]

\usepackage{comment}
\usepackage[dvipsnames, usenames]{color}
\usepackage[colorlinks]{hyperref}
\hypersetup{%
pdftitle={Preprint},
pdfauthor={Fan QIN},
pdfkeywords={Quantum Cluster Algebra, etc.},
bookmarksnumbered,
pdfstartview={FitH},
breaklinks=true,
urlcolor=blue,
citecolor=blue,
}%
\usepackage[all]{hypcap} 

\usepackage{breakurl}

\usepackage{color}

\newtheorem{Lem}[Thm]{Lemma}
\newtheorem{Prop}[Thm]{Proposition}
\newtheorem{Cor}[Thm]{Corollary}

\newtheorem{Quest}[Thm]{Question}
\newtheorem{Eg}[Thm]{Example}
\newtheorem{Rem}[Thm]{Remark}
\newtheorem{Def}[Thm]{Definition}

\newtheorem*{Def*}{Definition}
\newtheorem*{Thm*}{Theorem}

\newcommand{\Z}{\mathbb{Z}}
\newcommand{\N}{\mathbb{N}}

\newcommand{\C}{\mathbb{C}}
\newcommand{\R}{\mathbb{R}}

\newcommand{\ie}{{\em i.e.}\ }
\newcommand{\cf}{{\em cf.}\ }
\newcommand{\eg}{{\em e.g.}\ }

\newcommand{\st}{{such that}\ }

\renewcommand{\hat}[1]{\widehat{#1}}
\renewcommand{\tilde}[1]{\widetilde{#1}}

\newcommand{\opname}[1]{\operatorname{\mathsf{#1}}}

\renewcommand{\mod}{\opname{mod}}

\newcommand{\rep}{\opname{rep}}
\newcommand{\Rep}{\opname{Rep}}

\newcommand{\pr}{\opname{pr}}

\newcommand{\ind}{\opname{ind}}

\newcommand{\add}{\opname{add}}

\newcommand{\op}{^{op}}

\newcommand{\ra}{\rightarrow}
\newcommand{\xra}{\xrightarrow}

\newcommand{\Gr}{\opname{Gr}}

\newcommand{\dimv}{\underline{\dim}\,}

\newcommand{\rank}{\opname{rank}}

\newcommand{\Ker}{\opname{Ker}}
\renewcommand{\Im}{\opname{Im}}

\newcommand{\codim}{\opname{codim}}

\newcommand{\id}{\mathbf{1}}

\newcommand{\Hom}{\opname{Hom}}

\newcommand{\Ext}{\opname{Ext}}

\newcommand{\Hf}{{\frac{1}{2}}}

\newcommand{\Rm}[1]{{\longmapsto}}
\newcommand{\Lm}[1]{{\longmapsfrom}}

\newcommand{\cA}{{\mathcal A}}

\newcommand{\cC}{{\mathcal C}}
\newcommand{\cD}{{\mathcal D}}
\newcommand{\cE}{{\mathcal E}}
\newcommand{\cF}{{\mathcal F}}

\newcommand{\cK}{{\mathcal K}}
\newcommand{\cL}{{\mathcal L}}
\newcommand{\cM}{{\mathcal M}}

\newcommand{\cP}{{\mathcal P}}
\newcommand{\cQ}{{\mathcal Q}}
\newcommand{\cR}{{\mathcal R}}

\newcommand{\cT}{{\mathcal T}}
\newcommand{\cU}{{\mathcal U}}

\newcommand{\cY}{{\mathcal Y}}

\newcommand{\sT}{{\mathbb T}}

\newcommand{\tB}{{\widetilde{B}}}

\newcommand{\tQ}{{\widetilde{Q}}}

\newcommand{\tW}{{\widetilde{W}}}

\newcommand{\td}{{\widetilde{d}}}

\newcommand{\tg}{{\tilde{g}}}

\newcommand{\grRep}{{\Rep^\bullet}}

\newcommand{\pbw}{M}
\newcommand{\can}{L}
\newcommand{\gen}{\mathbb{L}}

\newcommand{\qChar}{\chi_{q,t}}
\newcommand{\tChar}{\chi_{q,t}}
\newcommand{\tqChar}{\hat{\qChar}}
\newcommand{\qtChar}{\hat{\tChar}}

\newcommand{\pbwDTorus}{{\pbw^\dTorus}}

\newcommand{\pbwTorus}{{\pbw^\torus}}
\newcommand{\canTorus}{{\can^\torus}}
\newcommand{\genTorus}{{\gen^\torus}}
\newcommand{\pbwCl}{{\pbw^\clAlg}}
\newcommand{\canCl}{{\can^\clAlg}}
\newcommand{\genCl}{{\gen^\clAlg}}

\newcommand{\pbwBasis}{{\{\pbw^\torus(w)\}}}
\newcommand{\canBasis}{{\{\can^\torus(w)\}}}%

\newcommand{\pbwClBasis}{{\{\pbwCl(w)\}}}
\newcommand{\canClBasis}{{\{\canCl(w)\}}}%
\newcommand{\genClBasis}{{\{\genCl(w)\}}}%

\newcommand{\redWSet}{{\mathcal{J}}}

\newcommand{\redPbwBasis}{{\{\pbw^\torus(w),w\in\redWSet\}}}
\newcommand{\redCanBasis}{{\{\can^\torus(w),w\in\redWSet\}}}%
\newcommand{\redGenBasis}{{\{\gen^\torus(w),w\in\redWSet\}}}%

\newcommand{\clPbwBasis}{{\{\pbwCl(w),w\in\redWSet\}}}%
\newcommand{\clGenBasis}{{\{\genCl(w),w\in\redWSet\}}}%
\newcommand{\clCanBasis}{{\{\canCl(w),w\in\redWSet\}}}%

\newcommand{\pbwTarg}{{\pbw^\targSpace}}

\newcommand{\genRedTarg}{{\gen^\redTargSpace}}

\newcommand{\KGp}{{\cK}}
\newcommand{\dualKGp}{{\cK^*}}
\newcommand{\quotKGp}{{\cR_t}}

\newcommand{\tBase}{{\Z[t^\pm]}}
\newcommand{\vBase}{{\Z[v^\pm]}}
\newcommand{\qBase}{{\Z[q^{\pm\Hf}]}}
\newcommand{\qBaseCoeff}{{\Z P[q^{\pm\Hf}]}}
\newcommand{\ZCoeff}{{\Z P}}

\newcommand{\redTargSpace}{{\cY}}
\newcommand{\targSpace}{{\hat{\cY}}} 

\newcommand{\torus}{\cT}
\newcommand{\dTorus}{{D(\torus)}} 

\newcommand{\projQuot}{{\cM}}
\newcommand{\affQuot}{{\projQuot_0}}
\newcommand{\lag}{{\cL}}
\newcommand{\grProjQuot}{{\projQuot^\bullet}}
\newcommand{\grAffQuot}{{\affQuot^\bullet}}
\newcommand{\grLag}{{\lag^\bullet}}
\newcommand{\grFib}{{\mathfrak{m}^\bullet}}

\newcommand{\grVb}{{\tilde{{\mathcal{Z}}^\bullet}}}

\newcommand{\regStratum}{{\affQuot^\mathrm{reg}}}
\newcommand{\grRegStratum}{{\grAffQuot^\mathrm{reg}}}
\newcommand{\cor}{{\opname{cor}}}
\newcommand{\contr}{{\hat{\Pi}}}
\newcommand{\ztB}{{\tB^z}}
\newcommand{\zInd}{{\ind^z}}
\newcommand{\zLambda}{{\Lambda^z}}
\newcommand{\clAlg}{{\cA}}
\newcommand{\qClAlg}{{\cA^q}}
\newcommand{\subQClAlg}{{\cA^q_{sub}}}

\newcommand{\coeffFree}{{^\phi}}
\newcommand{\pureCoeff}{{^f}}
\newcommand{\kerMod}{{^\sigma}}

\newcommand{\vtx}{{\opname{I}}}

\newcommand{\wtLess}{{<_w}}

\newcommand{\oOmega}{{\overline{\Omega}}}
\newcommand{\oh}{{\overline{h}}}

\newcommand{\lSp}{{\opname{L}}}

\newcommand{\grEndSp}{{\lSp^\bullet}}

\newcommand{\diag}{{\delta}}
\newcommand{\canStr}{{b}}
\newcommand{\genStr}{{c}}

\newcommand{\redCanStr}{{\coeffFree \canStr}}
\newcommand{\redGenStr}{{\coeffFree \genStr}}

\newcommand{\gr}{{\mathrm{gr}}}

\newcommand{\RHS}{{\mathrm{RHS}}} 
\newcommand{\LHS}{{\mathrm{LHS}}}

\newcommand{\tRes}{{\tilde{\mathrm{Res}}}}
\newcommand{\res}{{\mathrm{Res}}}

\newcommand{\eMatrix}{{{\cE'}}}
\newcommand{\dT}{{\td'}}
\newcommand{\dTW}{{\dT_W}}

\newcommand{\trunc}{{^{\leq 0}}}

\newcommand{\one}{^{(1)}}
\newcommand{\two}{^{(2)}}
\newcommand{\typeI}{{^{(i)}}}

\newcommand{\basis}{{\textbf{B}}}

\usepackage{mathrsfs}

\usepackage{tikz}
 \usetikzlibrary{positioning,shapes,shadows,arrows,snakes}

\pgfdeclarelayer{edgelayer}
\pgfdeclarelayer{nodelayer}
\pgfsetlayers{edgelayer,nodelayer,main}

\tikzstyle{none}=[inner sep=0pt]
\tikzstyle{black box}=[draw=black, fill=black!25]
\tikzstyle{white box}=[draw=black, fill=white]
\tikzstyle{black circle}=[circle,draw=black!50, fill=black!25]
\tikzstyle{red circle}=[circle,draw=red!50, fill=red!25]
\tikzstyle{blue circle}=[circle,draw=blue!50, fill=blue!25]
\tikzstyle{green circle}=[circle,draw=green!50, fill=green!25]
\tikzstyle{yellow circle}=[circle,draw=yellow!50, fill=yellow!25]

\begin{document}
\title[$qt$-characters, bases, and correction technique]{$t$-Analog of $q$-Characters, Bases of Quantum Cluster Algebras, and a Correction Technique}

\author{Fan QIN}
\address[]{Fan Qin, Mathematical Sciences Center, Tsinghua University,
  Beijing, China}
\email{qin.fan.math@gmail.com}

\begin{abstract}
We first study a new family of graded quiver varieties
together with a new $t$-deformation of the associated Grothendieck
rings. This provides the geometric foundations for a joint paper by Yoshiyuki
Kimura and the author. 

We further generalize the result of that paper to any acyclic quantum cluster algebra with arbitrary
nondegenerate coefficients. In particular, we obtain the generic basis,
the dual PBW basis,
and the dual canonical basis. The method consists in a correction
technique, which works for general
quantum cluster algebras.
\end{abstract}

\maketitle

\tableofcontents

\section{Introduction}
\label{sec:intro}

\subsection{Motivation}
\label{sec:motivation}

The theory of cluster algebras, which was invented by Fomin and
Zelevinsky \cite{FominZelevinsky02}, has its origin in Lie theory and
combinatorics. Since its very beginning, the theory
of (quantum) cluster algebras has been related to many areas, such as
Poisson Geometry, discrete dynamical systems, higher Techm\"{u}ller
spaces, combinatorics, commutative and non-commutative algebraic
geometries, and representation theory, \cf \cite{Keller12}.

One of the main motivations of the study of
quantum cluster algebras \cite{BerensteinZelevinsky05} is to provide an algebraic framework for the dual canonical bases of quantum groups. Therefore, bases of cluster algebra deserve to be put under
scrutiny. The following two questions arise naturally.

\begin{Quest}\label{quest:basis}
1) How to construct bases of (quantum) cluster algebras? 

2) What are the corresponding structure constants and transition matrices?
\end{Quest}

\begin{Quest}\label{quest:canonicalBasis}
1) Can we construct a ``dual canonical basis'' of the quantum cluster algebra, such that its
structure constants are positive and it contains all the quantum cluster
monomials?

2) Can we identify this basis with a subset of the dual canonical basis of certain quantum groups?
\end{Quest}

\begin{Rem}
  We expect this basis for acyclic quantum cluster algebra to be identified
  with a proper subset of the dual canonical basis, because the acyclic
  cluster
  algebras are ``smaller'' than general quantum groups, \cf \cite{GeissLeclercSchroeer10}.
\end{Rem}

Furthermore, by \cite{FominZelevinsky07} \cite{Tran09}, (quantum)
cluster monomials vary
little when the choice of the coefficient pattern changes. If such a
result holds for the bases as well, our research on the bases will be
largely simplified. Surprisingly, to best of the author's knowledge,
this problem seems has not been studied in literature yet.

\begin{Quest}\label{quest:coefficients}
  How does the bases data depend on the choice of the coefficients and
  quantization? Are they controllable?
\end{Quest}

The final motivation of this paper comes from the study of quantum
cluster characters. In \cite{Qin10}, the author defined quantum
cluster characters of the rigid objects of certain cluster categories,
and showed that these characters describe the quantum cluster
monomials of acyclic quantum cluster algebras. Nagao proposed a more
general formula to describe the quantum cluster variables of general
(quantum) cluster algebras based on the theory of non-commutative
Donaldson-Thomas invariants developed by Kontsevich and Soibelmann (up to some conjectures), \cf
\cite{KontsevichSoibelman08} \cite{Nagao10}. The following question is
natural from the representation theoretic point of view.

\begin{Quest}\label{quest:genericChar}
Can we extend the quantum cluster character to generic objects, such
  that we obtain a generic basis of the quantum
  cluster algebra containing all the quantum cluster monomials?
\end{Quest}

\subsection{Previous context, strategies, and results}

\subsubsection*{Previous context}
Despite the many successful applications of (quantum) cluster algebras to other
areas, the basis construction problems which motivate their
introduction remain largely open. 

We have limited knowledge of the constructions of bases of the classical cluster algebras. In the approach via preprojective algebra, \cf
\cite{GeissLeclercSchroeer10}, Gei\ss,
Leclerc, and Schro\"{e}r have shown that if $G$ is a semi-simple complex
algebraic group and $N\subset G$ a maximal nilpotent subgroup, then
the coordinate algebra $\C[N]$ admits a canonical classical cluster
structure whose coefficient type is specific. They further constructed
the generic basis of $\C[N]$, which contains the cluster monomials, and
identified it with Lusztig's dual semicanonical basis of $\C[N]$
\cite{Lusztig00}. As another approach, recently, Musiker, Schiffler,
and Williams constructed bases for classical cluster algebras arising from
unpunctured surfaces with coefficients whose exchange matrix is of full rank, \cite{MusikerSchifflerWilliams11} .

Our knowledge of the bases of the quantum cluster algebras is even
more limited. \cite{Lamp10} \cite{Lamp11} \cite{DingXu11}
\cite{HernandezLeclerc11} obtained partial results of Question \ref{quest:canonicalBasis} for quivers of
finite and affine type. Standard bases and partial results of
triangular bases have been obtained for acyclic quantum cluster,
\cf \cite{BerensteinFominZelevinsky05} \cite{Zelevinsky11talk} \cite{BerensteinZelevinsky2012}.

\subsubsection*{Deformed monoidal pseudo-categorification}
\label{sec:monoidalCategorification}
In this paper, we use (deformed) monoidal pseudo-categorifications to give a new
approach to bases of (quantum) cluster algebras.

Monoidal categorification was used by Hernandez and Leclerc as a
new approach to the positivity conjecture of cluster
  algebras, \cf \cite{HernandezLeclerc09}. For a given
cluster algebra $\clAlg$, we want to find
a tensor category $\cC$ together with a well-behaved algebra isomorphism from
$\clAlg$ to the Grothendieck ring $R$ of $\cC$. In particular, each
cluster monomial should be sent to the class of a simple module. In
the original work \cite{HernandezLeclerc09}, the tensor category $\cC$
is the tensor category of certain finite-dimensional modules
of certain quantum affine algebra, and the isomorphism is obtained by
comparing the (truncated) $q$-characters of these modules with cluster characters.

Therefore, we easily arrive at the following naive idea: in order to
construct the bases of (quantum) cluster algebras, it suffices to study bases of the (deformed) Grothendieck ring, and then apply the algebra isomorphism.

In this paper, we use a similar category $\cC$ (closely related to quantum loop
algebras) and construct a linear map identifying the
(truncated) $qt$-characters with quantum cluster characters studied by
the author in \cite{Qin10}. Our construction holds for general coefficients and
quantizations, where the linear map fails to be algebraic, thus the name
``pseudo-categorification''. A key ingredient is the observation that
the failures are mild and controllable. Therefore, in practice, we can still
follow the above
naive idea to study bases.

\subsubsection*{Graded quiver varieties}

In order to construct this monoidal pseudo-categorification, we need
to understand the deformed Grothendieck ring and the
$qt$-characters. A second key ingredient of the paper is that
we can understand them geometrically by using Nakajima's quiver
varieties, \cf \cite{Nakajima01} \cite{Nakajima04}.

Recall that Nakajima's quiver variety is a natural
generalization of the ADHM-construction of instantons. His graded
quiver variety is the fixed point subvariety with respect to a
$\C^*$-action. In \cite{Nakajima09}, he required the quiver to be
bipartite when constructing the graded quiver
varieties.

In Section \ref{sec:quiverVarieties}, we introduce a new family of graded
quiver varieties for acyclic quivers which are not necessarily
bipartite. In order to remove the bipartite restriction, we carefully change the definition of the
$\C^*$-action, \st the resulting new graded quiver varieties (fixed
point sets) still have good properties. Then we establish the
geometric foundations of our discussion, and proceed to study the deformed Grothendieck rings and
$qt$-characters following the arguments of Nakajima in
\cite{Nakajima01} \cite{Nakajima04} \cite{Nakajima09}.

Furthermore, as an important application, these constructions lead to an
affirmative answer of Question \ref{quest:canonicalBasis} for acyclic
quantum cluster algebras with specific coefficients and quantization, which will
be presented in the joint work \cite{KimuraQin11} by Yoshiyuki Kimura
and the author. In particular, by \cite[Theorem 3.3.7, Corollary 3.3.9]{KimuraQin11}, the result of Nakajima
\cite{Nakajima09} can be generalized as the following:

\textit{When the quantum cluster algebra contains an acyclic seed, the
  positivity conjecture is true. Namely, the coefficients of the
  cluster expansions of the quantum cluster
  variables are positive with respect to any chosen seed.}

Its proof depends on the Fourier-Deligne-Sato transform, \cf
the recent work by Efimov \cite{Efimov11} for an independent different
approach. We shall use
some useful results obtained in the proof.

\subsubsection*{Main results}

Via the pseudo-categorification approach introduced above, we use
$t$-analogue of $q$-characters to realize three bases of the acyclic quantum
cluster algebras for any choice of coefficients and quantizations: the
generic basis, the dual PBW basis, and the canonical basis, \cf
Theorem \ref{thm:genBasis} \ref{thm:otherBases}. For such
(quantum) cluster algebras, we obtain affirmative answers to Question \ref{quest:basis},
\ref{quest:canonicalBasis}(1), \ref{quest:genericChar}, \cf sections
\ref{sec:otherBases} \ref{sec:structureConstant}.

In the last section of the paper, we develop the correction technique
for general quantum cluster algebras, \cf Theorem
\ref{thm:correction}, which has been
implicitly used in previous sections when we measure the failures of
the monoidal categorification in \cite{KimuraQin11}. In the acyclic case, we obtain an affirmative answer to Question
\ref{quest:coefficients}, \cf sections
\ref{sec:otherBases} \ref{sec:structureConstant}.

\begin{Rem}
By the joint work of Kimura with the author \cite{KimuraQin11}, \ref{quest:canonicalBasis}(2) has an
  affirmative answer for acyclic quantum cluster algebra of a specific choice
  of coefficients and quantization. For a general choice, it seems unlikely that we can find a reasonable identification which preserves the multiplication. First, the number of the generators of the center of the quantum group restricts the possible
  number of the frozen vertices. Second, if we fix the initial seed of
  the given quantum cluster algebra and identify the initial variables with
  a collection of generators of the quantum group, the identification
  will fail to preserve the multiplication if we make any change to the
  quantization ($\Lambda$-matrix) of the initial seed. However, our
  observations do not rule out the possibility of establishing a new
  identification by choosing a different quantum group or different generators (of both algebras).
\end{Rem}

It should be mentioned that, despite that the results on the generic
quantum cluster characters and the quantum generic basis seem very
natural, their proof needs the full power of the graded quiver varieties
and the existence of the monoidal categorification. For the moment, the author does not see any alternative approach.

\subsection{Plan of the paper}
\label{sec:content}

In section \ref{sec:preliminaries}, we recall notations of
quantum cluster algebras arising from ice quivers.

In section \ref{sec:normalization}, we show how the cluster
expansions of quantum cluster variables (monomials) of quantum cluster
algebras with arbitrary compatible pairs can be deduced from those
with unitally compatible pairs.

In section \ref{sec:quiverVarieties}, we study a family of
new graded quiver varieties by choosing a new torus action. Then we
follow the statements of \cite{Nakajima09} to construct the
deformed Grothendieck ring $R_t$. 

In section \ref{sec:qtCharacters}, we define the $t$-analogue of
$q$-characters for these graded quiver varieties.

In section \ref{sec:pseudoModules}, we collect useful results from
\cite{KimuraQin11} on generic characters.

In section \ref{sec:failure}, we measure the failure of these
characters being isomorphism. Our constructions allow to quantize \cite{Nakajima09} by comparing the $qt$-characters with
  the quantum cluster characters in \cite{Qin10}. 

In section \ref{sec:bases}, we construct bases of acyclic quantum
cluster algebra for any choice of coefficients and quantizations. We
also study their transition matrices and structure constants.

Finally, in section \ref{sec:correction}, we establish the correction
technique for general quantum cluster algebras.

\section{Preliminaries}\label{sec:preliminaries}

We refer the readers to \cite{QinThesis} or \cite[Section 2]{KimuraQin11} for
details, whose conventions will be briefly recalled here.

\subsection{Quantum cluster algebras}

\label{sec:quantumClusterAlgebras}

We first recall the definition of a quantum cluster algebra, \cf \cite{BerensteinZelevinsky05}. Let $(\Lambda,\tB)$ denote a \emph{compatible pair}, namely we
have
\begin{align}
  \Lambda(-\tB)= \begin{bmatrix}
  D  \\
  0 
 \end{bmatrix},
\end{align}
where the \emph{$B$-matrix} $\tB$ is an
  $m\times n$ integer matrix and the \emph{$\Lambda$-matrix} $\Lambda$ is an
  $m\times m$ integer matrix for some integers $m\geq n$, and $D$ is a
  diagonal matrix with strictly positive integers on the diagonal. The \emph{principal
    part} $B$ of $\tB$ is defined to be its upper $n\times n$
  submatrix.

We use $v$ to denote the formal parameter $q^\Hf$, while $v^2$
is sometimes denoted by $q$. The \emph{quantum torus}
$\cT=\cT(\Lambda)$ associated with the $\Lambda$-matrix $\Lambda$ is the Laurent polynomial ring
$\vBase[x_1^\pm,\ldots,x_m^\pm]$, whose usual product $\cdot$ is often
omitted. Let $\qBaseCoeff$ denotes its subring
$\vBase[x_{n+1}^\pm,\ldots,x_m^\pm]$, which we call the\emph{
  coefficient ring}.

The matrix product $g^T \Lambda h$, $g,h\in \Z^m$, is denoted
by $\Lambda(g,h)$, where we use $g^T$ to denote the matrix
transposition of $g$. Then we endow $\cT$ with the twisted product $*$
\st we have
\[
x^{g}*x^{h}=v^{\Lambda(g,h)}x^{g+h}
\]
for any degrees $g$ and $h$ in $\Z^m$. 
The natural involution of $\cT$ which sends $v$ to $v^{-1}$ is
  denoted by $\overline{(\ )}$.

We fix an $n$-regular tree $\sT_n$ with root
$t_0$. Recursively, we can associate a \emph{quantum seed} $(\Lambda(t),\tB(t),x(t))$ with each vertex $t$ of the $n$-regular
tree \st we have
\begin{enumerate}
\item $(\Lambda(t_0),\tB(t_0),x(t_0))=(\Lambda,\tB,x)$, where
  $x=(x_1,\cdots,x_m)$, and
\item if there exists an edge labeled $k$ connecting two vertices $t$
  and $t'$, then the two quantum seeds $(\Lambda(t'),\tB(t'),X(t'))$ and
  $(\Lambda(t),\tB(t),X(t))$ are related by the \emph{mutation} at
  $k$, \cf \cite{KimuraQin11}.
\end{enumerate}

We define the $x$-variables to be $x_i(t)$, $1\leq i\leq m$,
$t\in\sT_n$. Those $x$-variables $x_i(t)$ with $1\leq i\leq n$ are
called the \emph{quantum cluster variables}. The \emph{quantum cluster
  monomials} are the monomials of the quantum cluster variables contained in any common
quantum seed.

Let $\cF$ denote the skew-field of fractions of the quantum torus
$\cT$. The \emph{quantum cluster algebra} $\qClAlg$ over $(R,v)$ is the
$(+,*)$-subalgebra of the skew-field generated by the $x$-variables
and the elements $x_{j}^{-1}$ for all $j>n$. According to the
\emph{Quantum Laurent phenomenon}, \cf \cite[Section
5]{BerensteinZelevinsky05}\label{thm:Laurent}, $\qClAlg$ is
contained in the quantum torus $\torus$.

When we specialize $v$ to
$1$, we obtain the \emph{classical
cluster algebra} $\clAlg=\qClAlg|_{v\mapsto 1}$, which we also denote by $\cA^\Z$.

\subsection{Coefficient types (frozen patterns)}
\label{sec:frozenQuiver}
We use $Q$ to denote a \emph{quiver }whose set of vertices is labeled by
$I=\{1, \ldots, n\}$. Any arrow $h$ of $Q$ points from its source $s(h)$ to its target $t(h)$. By
reversing the direction of $h$, we obtain a new arrow $\oh$. By
reversing all the arrow directions, we obtain the quiver $Q^{op}$.

Let $\tQ$ be a quiver with vertices vertices $\{ 1,\ldots,m\}$ which
contains $Q$ as a full sub-quiver. $Q$ is called the \emph{principal part} of
$\tQ$ and $\tQ$ is called an \emph{ice quiver}. The set of the \emph{frozen vertices}
$\{n+1,\ldots, m\}$ and all the arrows incident to them is
called the \emph{coefficient type} (or \emph{frozen
  pattern}) of the ice quiver $\tQ$, which we denote by $\tQ-Q$.

We associate with $\tQ$ an $m\times n$ integer matrix\footnote{Notice that this convention is opposite to that of \cite{Nakajima09}.} 
$\tB=(b_{ij})$ such that we have
\begin{align*}
  b_{ij}=\sharp\{h\in\tQ|s(h)=i,t(h)=j\}-\sharp\{h\in\tQ|s(h)=j,t(h)=i\}.
\end{align*}
Similarly, we associate with $Q$ an $n\times n$-matrix $B=B_Q$.

$Q$ is
called acyclic if it has no oriented cycles. We refer the reader to
\cite{KimuraQin11} or \cite{QinThesis} for the definition of
\emph{bipartite quivers} and \emph{level $l$ ice quivers with
$z$-pattern}. 

\begin{Eg}

(1) Figure \ref{fig:acyclicQuiver} is an example of acyclic quiver (which
is not bipartite). 

(2) Figure \ref{fig:levelOneQuiver} is an example of a level $1$ ice
quiver with z-pattern whose principal part is given by Figure \ref{fig:acyclicQuiver}. Notice that since $B_\tQ$ is invertible, its inverse provides
a canonical choice of $\Lambda$ \st $(\Lambda,\tB)$ is a compatible
pair.
\end{Eg}

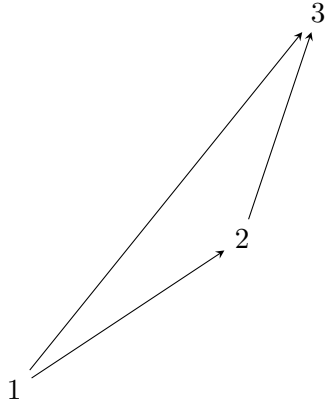
\begin{figure}[htb!]
 \centering
\beginpgfgraphicnamed{fig:acyclicQuiver}
  \begin{tikzpicture}
    \node [] (1) at (0,-4) {1}; 
    \node  [] (2) at (3,-2) {2}; 
    \node [] (3) at (4,1) {3};
    \draw[-stealth] (2) edge (3); 
    \draw[-stealth] (1) edge (3);
    \draw[-stealth] (1) edge (2);
  \end{tikzpicture}
\endpgfgraphicnamed
\caption{An acyclic quiver which is not bipartite}
\label{fig:acyclicQuiver}
\end{figure}

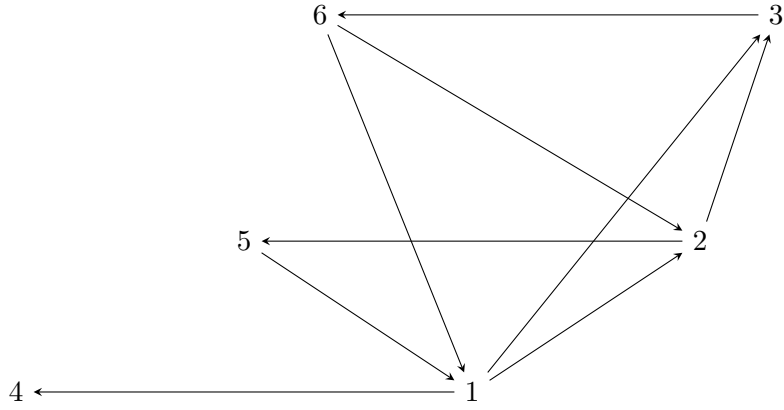
\begin{figure}[htb!]
 \centering
\beginpgfgraphicnamed{fig:levelOneQuiver}
         \begin{tikzpicture}
    \node [] (1) at (0,-4) {1}; 
    \node  [] (2) at (3,-2) {2}; 
    \node [] (3) at (4,1) {3};

\node [] (4) at (-6,-4) {4}; 
    \node  [] (5) at (-3,-2) {5}; 
    \node [] (6) at (-2,1) {6};

    \draw[-stealth] (1) edge (2); 
    \draw[-stealth] (2) edge (3); 
    \draw[-stealth] (1) edge (3);
    
    \draw[-stealth] (5) edge (1); 
    \draw[-stealth] (6) edge (1); 
    \draw[-stealth] (6) edge (2); 
    
    \draw[-stealth] (1) edge (4); 
    \draw[-stealth] (2) edge (5); 
    \draw[-stealth] (3) edge (6); 
       
  \end{tikzpicture}
\endpgfgraphicnamed
\caption{A level $1$ ice
quiver with z-pattern}
\label{fig:levelOneQuiver}
\end{figure}

\section{Quantum cluster variables via normalization}
\label{sec:normalization}

  Let $\tQ$ be an ice quiver with principal part $Q$ and $\tB$
  its associated $m\times n$ matrix. Let $\qClAlg$ be the quantum
  cluster algebra associated with a compatible pair $(\Lambda,\tB)$
  \st $\Lambda(-\tB)=\begin{bmatrix}D\\0 \end{bmatrix}$.

Following \cite{Amiot09} \cite{Plamondon10a}, we choose a
\emph{generic potential} $\tW$ associated with the ice quiver $\tQ$
and construct the corresponding \emph{cluster category} $\cC_{(\tQ,\tW)}$,
which will be denoted by $\cC_\tQ$ for simplicity. Then
to each $1\leq i\leq m$ and $t\in \cT_n$, we can associate an object
$M_i(t)$ in the cluster category. Furthermore, we
have a map $\ind(\ )$ sending each object in $\cC_\tQ$ to its index
$\ind(M_i(t))$ in $\Z^m$ such that $\ind(M_i(t_0))=e_i$. We refer the
reader to \cite[Section 2.2.2]{QinThesis} for details.

 First choose $\tQ-Q$ and $\Lambda$ \st $D$ is the identity matrix
 $\id_n$. Because the quantum $F$-polynomials exist, \cf
 \cite[Theorem 5.3]{Tran09}, the quantum cluster
  variables take the form
  \begin{align}\label{eq:qCCFormula}
    x_{i}(t)=\sum_{v\in\N^n}c_{i,v}(q^\Hf)x^{\ind(M_i(t))+\tB v},
  \end{align}
  for some bar-invariant Laurent polynomials
  $c_{i,v}\in\Z[t^\pm]$. For general $\tQ-Q$ and $\Lambda$, we have
  the following result.

\begin{Thm}\label{thm:normalizedQCCFormula}
Assume $D$ is a matrix whose diagonal entries are equal to
$\diag\in\Z_{>0}$. Then in $\qClAlg$, we have
\begin{align}
  x_{i}(t)=\sum_{v\in\N^n}c_{i,v}(q^{\frac{\diag }{2}})x^{\ind(M_i(t))+\tB v}.
\label{eq:normalizedQCCFormula}
\end{align}
\end{Thm}
\begin{proof}
Consider the $2n\times 2n$-matrix
\begin{align*}
  \begin{pmatrix}
B & -\id_n\\
\id_n & 0
\end{pmatrix}.
\end{align*} Denote its inverse by $-\Lambda_1$. Its submatrix $\tB'=\begin{pmatrix}
B \\
\id_n 
\end{pmatrix}$ is the $2n\times n$-matrix associated with the \emph{ice
quiver with principal coefficient} $\tQ'$ and principal part $Q$. We
have the objects
$M'_i(t)$ in the cluster category $\cC_{\tQ'}$, and the map $\ind'$.

Notice that $(\Lambda_1,\tB')$ is a unitally
compatible pair, and we have the associated quantum cluster algebra $\qClAlg_1$.

Inspired by \cite[Proposition 5.1]{Tran09}, we consider the $2n\times 2n$-matrix
\begin{align*}
  \Lambda_2=\begin{pmatrix}
0 & -D\\
D & BD
\end{pmatrix}.
\end{align*}
The pair $(\Lambda_2,\tB')$ is again compatible. In fact, one
  can check that $\Lambda_2=\diag\cdot\Lambda_1$. Denote the
associated quantum cluster algebra by $\qClAlg_2$.

Since $\Lambda_2$ is divisible by $\diag$, we can view $\qClAlg_2$ as a
quantum cluster algebra over $(\Z[v^\pm],v)$ with
$v=q^{\frac{\diag }{2}}$ and the initial compatible pair is $(\Lambda_2/\diag,\tB')$. Then by sending $q^\Hf$ to $q^{\frac{\diag }{2}}$,
$\qClAlg_1$ is identified with $\qClAlg_2$. In $\qClAlg_2$, we have
\begin{align*}
  x_{i}(t)=&\sum_{v\in\N^n}c_{i,v}(q^{\frac{\diag
    }{2}})x^{\ind'(M'_i(t))+\tB' v}\\
=&(\sum_vc_{i,v}(q^{\frac{\diag }{2}})q^{-\Hf\Lambda_2(\tB' v,
    \tg_{i,1})}x^{\tB' v})*x^{\ind'(M'_i(t))}
\end{align*}
for the vectors $\tg_{i,1}=\ind'(M'_i(t))$.

Finally, by \cite[Theorem 5.3]{Tran09} quantum $F$-polynomials
exist. Therefore, the quantum cluster variables in $\qClAlg$ can be
written as
\begin{align*}
  x_{i}(t)=&(\sum_vc_{i,v}(q^{\frac{\diag }{2}})q^{-\Hf\Lambda_2(\tB' v,
    \tg_{i,1})}x^{\tB v} )*x^{\ind(M_i(t))}\\
=&\sum_{v\in\N^n} q^{\Hf f_{i,v}}c_{i,v}(q^{\frac{\diag}{2}})x^{\ind(M_i(t))+\tB v}.
\end{align*}
for some integers $f_{i,v}\in\Z$. Notice that $\{X^{\tg_i+\tB
  v}|v\in\N^n\}$ is linearly independent over $\qBase$ because we can show that $\tB$ is full rank. Since $X_i(t)$ and $c_{i,v}$ are
bar-invariant, all $f_{i,v}$ must vanish.
\end{proof}

\section{Graded quiver varieties}\label{sec:quiverVarieties}
This section contains proofs of properties of our new graded quiver varieties,
which are needed by \cite{KimuraQin11} and later sections of this paper.

We begin with constructing some graded quiver
varieties associated with acyclic quivers. Our graded quiver varieties
are slightly different from the original ones in
\cite{Nakajima01}. Thanks to our modified definition, we can show the existence of suitable stratifications for arbitrary acyclic quivers
(Proposition \ref{prop:grStratification}). To prove this and other
geometric properties, we adapt Nakajima's original arguments in \cite{Nakajima98},
\cite{Nakajima01}.

\begin{Rem}
It seems likely that our definition should still enjoy the link with quantum
loop algebras established in \cite{Nakajima01}. However, since Lie theory is
not the major focus in this paper, we do not recheck the
link. Therefore, when we follow any useful arguments in
\cite{Nakajima01}, the Lie theoretic expressions should be replaced with
their geometric counterparts.
\end{Rem}
\subsection{Graded quiver varieties}
\label{sec:gradedQuiverVarieties}

We first give the definitions of the graded quiver varieties.

Fix a quiver $Q$ which is acyclic, \cf Section
\ref{sec:frozenQuiver}. We label its vertices such that there is no
arrows from $i$ to $j$ if $i\geq j$. The associated Cartan matrix
$C=(c_{ij})$ is given by
\begin{align}
c_{ij}=\left\{
\begin{array}{cc}
2 & \text{if $i= j$}\\
-|b_{ij}| & \text{if $i\neq j$}
\end{array} \right. .
\end{align}

Let $C_q$ denote the linear map from $\Z^{I\times (\Hf+\Z)}$ to
$\Z^{I\times \Z}$ such that for any $\xi\in \Z^{I\times (\Hf+\Z)}$,
the image $C_q\xi$ is given by
\begin{align}
(C_q \xi)_k(a)=\xi_k(a-\Hf)+\xi_k(a+\Hf)-\sum_{j:k< j\leq n}b_{kj}\xi_j(a-\Hf)-\sum_{i:1\leq i< k}b_{ik} \xi_i(a+\Hf).
\end{align}
 
\begin{Def}[$q$-Cartan matrix]\label{def:cartanMatrix}
The map $C_q$ is the \emph{$q$-analogue} of the Cartan matrix $C$.
\end{Def}

The map $C_q$ induces a map from $\Z^{I\times \R}$
to $\Z^{I\times \R}$, which we denote by $C_q$ by abuse of notation. It follows that
$C_q[d]$ equals $[d]C_q$ for any $d\in \R$, where $[d]$ is the natural
degree shift.

\begin{Lem}\label{lem:symmetricCartan}
Let $\xi^1,\xi^2\in \Z^{I\times \R}$ be any two given vectors, such
that at least one of them has finite support, then we have
  \begin{align}
    \label{eq:symmetricForm}
    \xi^2 \cdot C_q\xi^1[-\Hf]=C_q\xi^2\cdot \xi^1[-\Hf].
  \end{align}
\end{Lem}

\begin{proof}
  To prove the equation, it suffices to expand both sides:
  \begin{align*}
    \LHS&=\sum_{k,a}\xi_k^2(a)\cdot(\xi_k^1(a)+\xi_k^1(a-1)-\sum_{i:i<k}b_{ik}\xi_i^1(a)-\sum_{j:k<j}b_{kj}\xi_j^1(a-1))\\
    &=\sum_{k,a}(\xi_k^2(a-\Hf)+\xi_k^2(a+\Hf)-\sum_{j:k<j}b_{kj}\xi_j^2(a-\Hf)\\
&\qquad-\sum_{i:i<k}b_{ik}\xi_i^2(a+\Hf))\cdot\xi_k^1(a-\Hf)\\
    &=\RHS.
  \end{align*}
\end{proof}

Notice that there is a lexicographical order on the set $I\times\R$,
which is also called the \emph{weight order} $\wtLess$ in \cite{QinThesis}.

Define $E$ to be $\{\xi\in \Z^{I\times \Z}|\xi_k(a)=0,\ \forall k\in
I,\ a\ll 0\}$.
\begin{Lem}
  \begin{enumerate}
  \item $C_q$ induces a bijection between $E[\Hf]$ and $E$.
 
    \item $C_q$ induces a bijection between $E$ and $E[-\Hf]$.
  \end{enumerate}
\end{Lem}
\begin{proof}
Since $C_q$ commutes with $[\Hf]$, it suffices to show 1).  Consider
the block matrix of the restricted map $C_q[\Hf]$ from $E$ to $E$,
whose components are indexed by $(i,a)\times(j,b)$, $i,j\in I$,
$a,b\in \Z$. It is lower unitriangular with respect to the lexicographical
order on $I\times \R$. Notice that $E$ is bounded below. The claim follows.
\end{proof}
Therefore, it makes sense to talk about the inverses of the above
bijections, which are both denoted by $C_q^{-1}$.

Next, we generalize the graded quiver varieties of \cite{Nakajima09}
from the bipartite to the acyclic case, modifying the original construction via the lexicographical order. We follow the convention of
\cite{QinThesis} and consider bigraded dimension vectors $w=(w_i(a))_{i\in I,
  a\in \Z}$ and $v=(v_i(b))_{i\in I,b\in {\Z+\Hf}}$. We always assume
that they have non-negative components and finite supports. Let
$W=\C^w=\oplus_{i,a}W_i(a)$ and $V=\C^v=\oplus_{i,b} V_i(b)$ be the
associated bigraded vector spaces. If two $I\times \R$-graded vectors are given, such that at least
one of them has finite support, their scalar product $\cdot$
is well defined.

We say a pair $(v,w)$ is \emph{$l$-dominant} if the difference $w-C_q
v$ is contained in $\N^{I\times \Z}$. We say $(v,w)$ dominates
$(v',w')$, which we denote by $(v,w)\geq
(v',w')$ (\emph{dominance order}), if there exists some $v''\in
\N^{I\times(\Z+\Hf)}$ such that $w'-C_q v'=w-C_q (v+v'')$. We denote $w'\leq w$ if $(0,w')\leq (0,w)$.

For any $v$, $v'\in\N^{I\times(\Z+\Hf)}$, $w\in\N^{I\times \Z}$, which have finite
supports, we define
\begin{align*}
\opname{L}^\bullet(v,v')&=\oplus_{(i,b)} \opname{Hom}(V_i(b),V'_i(b)),\\
\opname{L}(w,v)&=\oplus_{(i,a)} \opname{Hom}(W_i(a),V_i(a-\Hf)),\\
\opname{L}(v,w)&=\oplus_{(i,b)}  \opname{Hom}(V_i(b),W_i(b-\Hf)),\\
 \opname{E}(v,v')&=(\oplus_{h\in\Omega,b} \opname{Hom}(V_{s(h)}(b),V'_{t(h)}(b)))\oplus (\oplus_{\oh\in\oOmega,b} \Hom(V_{s(\oh)}(b),V'_{t(\oh)}(b-1))),
\end{align*}
where $\Omega$ and $\oOmega$ are the sets of arrows of the quivers
$Q\op$ and $Q$ respectively. Let $H$ be the union of $\Omega$ and
$\oOmega$, and $\epsilon$ the function on $H$ such that it sends
$\Omega$ to $1$ and $\oOmega$ to $-1$ respectively. The vector space
$\opname{E}(v,v')$ consists of the elements $(B_h)_{h\in H}$, where
$B_h=\oplus_{b\in\Z+\Hf}B_{h,b}$ and $B_{h,b}$ is any linear map from $V_{s(h)}(b)$ to $V'_{t(h)}(b+\frac{\epsilon(h)-1}{2})$. The function $\epsilon$ acts on $B_h$
by $\epsilon B_h=\epsilon(h)B_h$. 

Define the vector space
\begin{equation}
\grRep(Q\op,v,w)=\opname{E}(v,v)\oplus\opname{L}(w,v) \oplus \opname{L}(v,w),
\end{equation}
whose points are given by 
\begin{equation}
  \begin{split}
(B,\alpha,\beta)=&((B_{h})_{h\in H},\alpha,\beta)\\=&((b_h)_{h\in\Omega}, (b_{\oh})_{\oh\in\oOmega},(\alpha_i)_i,(\beta_i)_i)\\=&((\oplus_b b_{h,b})_{h\in\Omega},(\oplus_b b_{\oh,b})_{\oh\in\oOmega},(\oplus_a\alpha_{i,a})_i,(\oplus_b\beta_{i,b})_i).
  \end{split}
\end{equation}
The restriction of the moment map for ungraded quiver varieties
becomes the map $$\mu: \grRep(Q\op, v,w)\ra \grEndSp(v,v[-1])$$ such
that we have
\begin{align}
    \mu(B,\alpha,\beta)=\oplus_{i,b}(\sum_{h\in\Omega}(b_{h,b} b_{\oh,b+1}-b_{\oh,b+1}
    b_{h,b+1}) +\alpha_{i,b+\Hf} \beta_{i,b+1}).
\end{align}

\begin{Eg}
Figure
\ref{fig:repSpace} provides an example of $\grRep(Q\op,v,w)$, whose
rows and columns are indexed by $I$-degrees and $\R$-degrees respectively.
\end{Eg}

\begin{figure}[htb!]
 \centering
\beginpgfgraphicnamed{repSpace}
\begin{tikzpicture}[scale=0.75]

\node [] (deg0) at (5,4) {$\mathrm{deg}=-\frac{1}{2}$};
\node [] (v1) at (4,-5) {$V_1(-\frac{1}{2})$};
\node [] (v2) at (6,-2) {$V_2(-\frac{1}{2})$};
\node [] (v3) at (6,2) {$V_3(-\frac{1}{2})$};

\draw[-triangle 60] (v2) edge node[above] {$h$} (v1);
\draw[-triangle 60] (v3) edge (v2);
\draw[-triangle 60] (v3) edge (v1);

\node [] (deg1) at (-3,4) {$\mathrm{deg}=-\frac{3}{2}$};
\node [] (v4) at (-4,-5) {$V_1(-\frac{3}{2})$};
\node [] (v5) at (-2,-2) {$V_2(-\frac{3}{2})$};
\node [] (v6) at (-2,2) {$V_3(-\frac{3}{2})$};
\node [] (dotL) at (-4,-2) {$\ldots$};

\draw[-triangle 60] (v5) edge (v4);
\draw[-triangle 60] (v6) edge (v5);
\draw[-triangle 60] (v6) edge (v4);

\node [color=blue] (deg2) at (9,4) {$\mathrm{deg}=0$};
\node [color=blue] (w1) at (7,-5) {$W_1(0)$};
\node [color=blue] (w2) at (9,-2) {$W_2(0)$};
\node [color=blue] (w3) at (9,2) {$W_3(0)$};
\node [] (dotR) at (10.5,-2) {$\ldots$};

\node [color=blue] (deg3) at (1,4) {$\mathrm{deg}=-1$} ;
\node [color=blue] (w4) at (0,-5) {$W_1(-1)$};
\node [color=blue] (w5) at (2,-2) {$W_2(-1)$};
\node [color=blue] (w6) at (2,2) {$W_3(-1)$};

\draw[-triangle 60] (w1) edge  node[above] {$\alpha_1$} (v1);
\draw[-triangle 60] (w2) edge (v2);
\draw[-triangle 60] (w3) edge (v3);

\draw[-triangle 60] (v1) edge node[above] {$\beta_1$} (w4);
\draw[-triangle 60] (v2) edge (w5);
\draw[-triangle 60] (v3) edge (w6);

\draw[-triangle 60] (w4) edge (v4);
\draw[-triangle 60] (w5) edge (v5);
\draw[-triangle 60] (w6) edge (v6);

\draw[-triangle 60] (v1) edge node[above] {$\overline{h}$}  (v5);
\draw[-triangle 60] (v1) edge (v6);
\draw[-triangle 60] (v2) edge (v6);

\end{tikzpicture}
\endpgfgraphicnamed
\caption{Vector space $\grRep(Q\op,v,w)$}
\label{fig:repSpace}
\end{figure}

The base change group $G_v=\prod_{i,a} GL(V_{i,a})$ naturally acts on
$\mu^{-1}(0)$. Define $\chi$ to be the character which sends any group
element $g$ to $\prod_{i,a}(\det g_{i,a})^{-1}$.

Let $\mu^{-1}(0)^s$ denote the set of $\chi$-stable points in
$\mu^{-1}(0)$ and $\grProjQuot(v,w)$ the free quotient $\mu^{-1}(0)^s/G_v$. This is a quasi-projective variety. Define $\grAffQuot(v,w)$ to be the affine variety
$\opname{Spec}(\C[\mu^{-1}(0)]^{G_v})$. Let $\pi$ denote the canonical
projective morphism from $\grProjQuot(v,w)$ to $\grAffQuot(v,w)$. For
any point $x$ in $\grAffQuot(v,w)$, denote $\pi^{-1}(x)$ by
$\grFib_x(v,w)$. We also denote $\pi^{-1}(0)=\grLag=\grLag(v,w)$. The varieties
$\grProjQuot(v,w)$, $\grAffQuot(v,w)$, $\grLag(v,w)$ are called \emph{graded
quiver varieties}.

In the rest of this section, we verify important properties of graded quiver varieties.

\subsection{Ungraded quiver varieties}
\label{sec:ungradedQuiverVarieties}
We first recall important properties of ungraded quiver varieties.

Let $V$ and $W$ be finite-dimensional $I$-graded complex vector spaces
(without $\Z$-grading).  In analogy with the previous subsection, we
have vector spaces
\begin{align*}
L(V,V)=\oplus_i\Hom(V_i,V_i)\\
L(W,V)=\oplus_i\Hom(W_i,V_i)\\
L(V,W)=\oplus_i\Hom(V_i,W_i)\\
E(V,V)=\oplus_{h\in H}\Hom(V_{s(h)},V_{t(h)}).
\end{align*}
Consider the symplectic vector space
$\Rep(Q\op,V,W)=L(W,V)\oplus L(V,W)\oplus E(V,V)$. The associated moment map
$\mu:\Rep(Q\op,V,W)\ra L(V,V)$ takes a point $(B,\alpha,\beta)$ of
$\Rep(Q\op,V,W)$ to
\begin{align*}
 \mu(B,\alpha,\beta)= (\epsilon B)B +\alpha \beta.
\end{align*}
Following the arguments of \cite{Nakajima98}, we consider the $GL(V)$-variety $\mu^{-1}(0)$, and fix the character $\chi$ of $GL(V)$ \st
$\chi(g)=\prod_i (\det g_i)^{-1}$. Then we can construct the geometric
invariant theory quotient (GIT quotient for short) $\projQuot(V,W)$ with respect to $\chi$ and the
categorical quotient $\affQuot(V,W)$ by the action of $GL(V)$ together with the projective morphism $\pi:\projQuot(V,W)\ra \affQuot(V,W)$.

The points in
$\affQuot(V,W)$ are in bijection with the closed orbits in
$\mu^{-1}(0)$. A point $(B,\alpha,\beta)$ in such a closed orbit is
called a representative of the corresponding point in $\affQuot(V,W)$,
which is denoted by $[B,\alpha,\beta]$. Let $\mu^{-1}(0)^s$ be the open subset
of $\mu^{-1}(0)$ consisting of the $\chi$-stable points. It is well
known that $GL(V)$ acts freely on $\mu^{-1}(0)^s$. Therefore, the
points of $\projQuot(V,W)$ are in bijection with the orbits of $\mu^{-1}(0)^s$. Again, a point $(B,\alpha,\beta)$ of such a free orbit is called a representative of the corresponding point in  $\projQuot(V,W)$, denoted by $[B,\alpha,\beta]$.

\begin{Prop}[{\cite[Corollary 3.12]{Nakajima98}}]\label{prop:smoothQuot}
    The variety $\projQuot(V,W)$ is smooth.
\end{Prop}

Given any two vectors $v$, $v'$ \st $v\leq v'$ (with respect to the coordinate-wise order), there is a natural embedding of
$\affQuot(V,W)$ into $\affQuot(V',W)$ given by extending the coordinates of
the representatives by zero. Take all possible $v$ and define
$\affQuot(W)=\cup_V \affQuot(V,W)$ to be the direct limit of all the embeddings. It is possibly infinite-dimensional, \cf
\cite[2.5]{Nakajima01}.

Let $[B,\alpha,\beta]$ be a point in $\projQuot(V,W)$ and let
$x=(B,\alpha,\beta)$ be its representative. Suppose that we have a $B$-invariant filtration of $V$
\begin{align*}
  0\subset F^0\subset F^1\subset \ldots\subset F^t=V,
\end{align*}
where $\Im \alpha\subset F^0$. Let $\gr_0\alpha$ denote
the morphism from $W$ to $F^0$ \st its composition with the inclusion
$F^0\ra V$ is $\alpha$. Let $\gr_0\beta$ denote the restriction of
$\beta$ to $F^0$. For $1\leq s\leq t$, let $\gr_s B$ denote the
endomorphism which $B$ induces on $F^s/F^{s-1}$ and $\gr
B=\oplus_{1\leq s\leq t}\gr_s B$ the endomorphism on $\oplus_{1\leq
  s\leq t}F^s/F^{s-1}$. The induced representative $\gr x$ is defined
to be $(\gr B,\gr_0 \alpha,\gr_0 \beta)$, \cf \cite[Definition 3.19]{Nakajima98}.

\begin{Prop}\cite[Proposition 3.20]{Nakajima98}
  \label{prop:grRepresentative}
Let $[B,\alpha,\beta]$ be a point in $\projQuot(V,W)$ and let
$x=(B,\alpha,\beta)$ be its representative. Then there exists a $B$-invariant filtration of $V$
\begin{align*}
  0\subset F^0\subset F^1\subset \ldots\subset F^t=V,
\end{align*}
\st $\Im \alpha\subset F^0$ and the induced triple $\gr x=(\gr
B,\gr_0\alpha,\gr_0\beta)$ is a representative of $\pi([B,\alpha,\beta])$.
\end{Prop}

If $\hat{G}$ is a subgroup of $G_v$, we denote by $(\hat{G})$ the
conjugacy class of $\hat{G}$. There is a natural stratification $
\affQuot(V,W)=\sqcup_{(\hat{G})} \affQuot(V,W)_{(\hat{G})}$, \st each
stratum is the set of the points $[B,\alpha,\beta]$ which have representatives $(B,\alpha,\beta)$ with the stabilizers in the conjugacy class $(\hat{G})$. 

\begin{Prop}[3.27,\cite{Nakajima98}]\label{prop:stratum}
Let $[B,\alpha,\beta]$ be a point in $\affQuot(V,W)_{(\hat{G})}$ for some nontrivial $\hat{G}$. Then there is a representative $(B,\alpha,\beta)$ and a decomposition 
\[
V=V^0\oplus (V^1)^{\oplus\hat{v_1}}\oplus \cdots \oplus (V^r)^{\oplus\hat{v_r}},
\]
\st we have
\begin{enumerate}
\item  $B(V^s)\subset V^s$ for each summand $V^s$, $0\leq s\leq r$;
\item if $s\neq s'$, there is no isomorphism from $V^s$ to $V^{s'}$ that commutes
  with $B$;
\item $\Im \alpha$ is contained in $V^0$, and $V^s$ is contained in $\Ker \beta$ for all $s>0$;
\item the restriction of $(B,\alpha,\beta)$ to $V^0$ has the trivial stabilizer in $\prod_{i\in I} GL(V_i ^0)$;
\item the subgroup $\prod_{i\in I} GL(V_i^s)$ meets $\hat{G}$ only in the scalar subgroup $\C^*\subset GL(V)$;
\end{enumerate}
\end{Prop}

\begin{Rem}\label{rem:stratum}
Restricting the equation $\mu(B,\alpha,\beta)=0$ to each summand
$V^s$, $s>0$, we see that $V^s$ is a module over the preprojective
algebra associated with $Q\op$ with respect to the restriction of the $B$-action. It has the minimal possible stabilizer $\C^*$ and a closed orbit under the $\prod_i GL(V_i^s)$ action. Therefore, it is a representative of a point in the categorical quotient $\affQuot(V^s,0)$.
\end{Rem}

We call $\affQuot(V,W)_{(\{1\})}$ the regular stratum and denote it by
$\regStratum(V,W)$. It is known that the restriction of $\pi$ gives an
isomorphism from $\pi^{-1}(\regStratum(V,W))$ to $\regStratum(V,W)$.

Assume $x$ is a point in $\regStratum(V^0,W)$, which is naturally embedded into a quotient $\affQuot(V,W)$. Let $T$ be the tangent space of $\regStratum(V,W)$ at $x$. Since $\regStratum(V^0,W)$ is non-empty, $(V^0,W)$ is $l$-dominant. Define $W^\bot=\C^{\dimv W-C \dimv V^0}$, $V^\bot=\C^{\dimv V-\dimv V^0}$. We have the following theorem.

\begin{Thm}[Theorem 3.3.2, \cite{Nakajima04}]\label{thm:transversalSlice}
There exist neighborhoods $\cU$, $\cU_\cT$, $\cU^\bot $ of $x\in \affQuot(V,W)$, $0\in \cT$, $0\in \affQuot(V^\bot ,W^\bot )$ respectively, and biholomorphic maps $\cU\ra \cU_T\times \cU^\bot $, $\pi^{-1}(\cU)\ra \cU_\cT\times \pi^{-1}(\cU^\bot )$, \st the following diagram commutes:
\begin{equation*}
\begin{CD}
  \projQuot(V,W) \;@. \supset \; @. \pi^{-1}(\cU) @>>\cong> \cU_T \times \pi^{-1}(\cU^\bot ) \;@.\subset \;@. T\times \projQuot(V^\bot ,W^\bot )\\
   @. @. @V{\pi}VV @VV{\id\times\pi}V @.@. \\
   \affQuot(V,W)\;@.\supset \; @. \cU @>>\cong> \cU_\cT\times \cU^\bot  \;@. \subset \;@. \cT\times\affQuot(V^\bot ,W^\bot )
\end{CD}
\end{equation*}
\end{Thm}

For any two $I$-graded vector spaces $V$, $V'$, define $L(V,V')=\oplus_i \Hom(V_i,V_i')$, $E(V,V')=\oplus_h \Hom(V_{s(h)}, V_{t(h)}')$.

For a point $(B,\alpha,\beta)\in \mu^{-1}(0)$, we consider the following complex
\begin{align}
  \label{eq:tangentComplex}
  L(V,V)\xra{\iota} E(V,V)\oplus L(W,V)\oplus L(V,W)\xra{d\mu} L(V,V),
\end{align}
where $d\mu$ is the differential map of $\mu$, and for $\xi\in L(V,V)$, we have
\begin{align}
  \label{eq:iota}
  \iota(\xi)=(\oplus_{h\in \Omega}(B_h\xi-\xi B_h))\oplus (-\xi \alpha) \oplus (\beta \xi).
\end{align}
Then the tangent space of $\projQuot(V,W)$ at $[B,\alpha,\beta]$ is isomorphic to the middle cohomology of this complex.

\subsection{Graded quiver varieties as fixed point sets}

In this subsection, we check the properties of our new graded quiver
varieties by using the technique developed by Nakajima in his previous works. The
results obtained in this subsection provide the geometric foundations
for \cite{KimuraQin11}.

Let $V$ and $W$ be as in Section \ref{sec:ungradedQuiverVarieties}. Using the fixed point technique developed in \cite{Nakajima01}, we
can deduce the properties of our graded quiver varieties from their
ungraded version in Section \ref{sec:ungradedQuiverVarieties}.

Choose a torus action\footnote{The choice is not unique. In fact,
  different choices might give isomorphic graded quiver varieties.} of
$\C^*$ on $\Rep(Q\op,V,W)$, \st for any $\varepsilon\in \C^*$, we have
\[
\varepsilon (\alpha,\beta,b_h,b_{\oh})=(\varepsilon^{n}\alpha, \varepsilon^{n}\beta,\varepsilon ^{2(s(h)-t(h))} b_h, \varepsilon ^{2(n-s(h)+t(h))} b_{\oh}).
\]

Taking into account of the natural action of $GL(W)$ on $\Rep(Q\op, V,W)$, we get a group action of $GL(W)\times \C^*$ on this space.

Since the actions of $GL(W)\times \C^*$ and $GL(V)$ commute, $GL(W)\times \C^*$ acts on the quiver varieties $\projQuot(V,W)$ and $\affQuot(V,W)$ and it commutes with the projective morphism $\pi$.

Take a pair $(s,\varepsilon)$, \st $s\in GL(W)$ is semisimple, and
$\varepsilon$ is not a root of unity. It generates a cyclic subgroup,
whose closure with respect to the Zariski topology is denoted by
$A$. As in \cite{Nakajima01}, let $[B,\alpha,\beta]$ be a point in $\projQuot(V,W)^A$ and $(B,\alpha,\beta)$ be any representative of
it. There exists a group homomorphism $\rho_{(B,\alpha,\beta)}$ from $A$ to $GL(V)$ \st for any element $a\in A$, we have
$a(B,\alpha,\beta)=\rho_{(B,\alpha,\beta)}(a)^{-1}(B,\alpha,\beta)$. The
conjugacy class of $\rho_{(B,\alpha,\beta)}$ is independent of the
choice of the representative $(B,\alpha,\beta)$ because the actions of $GL(V)$ and $A$
commute. So we can denote it by $[\rho_{[B,\alpha,\beta]}]$.

\begin{Lem}[{\cite[Section 4.1]{Nakajima01}}]
The map from $\projQuot(V,W)^A$ to the conjugacy classes of the group
homomorphisms from $A$ to $GL(V)$, sending $[B,\alpha,\beta]$ to
$[\rho_{[B,\alpha,\beta]}]$, is locally constant.
\end{Lem}
\begin{proof}
Since $A$ is generated by the element $a=(s,\varepsilon)$, it suffices
to study the conjugacy class $[\rho_{[B,\alpha,\beta]}(a)]$.

First, we show that $[\rho_{[B,\alpha,\beta]}(a)]$ is continuous in $[B,\alpha,\beta]$. Recall that $\mu^{-1}(0)^s$ is a principal $GL(V)$-bundle over
$\projQuot(V,W)$. Denote the fibre map by $p$. Take a
trivialization. Let $U$ be any chart. For any continuous curve $[x(t)]$,
$0\leq t\leq 1$, the curve $[x^U(t)]=U\cap [x(t)]$ in $U\cap\projQuot(V,W)^A$ can be lifted
to a continuous curve $x^U(t)= [x^U(t)]\times \{e\}$ in $(U\times GL(V)) \cap
p^{-1}(\projQuot(V,W)^A)$, where $e$ denotes the identity of
$GL(V)$. Then the fibre coordinates of the
continuous curve $a^{-1}x^U(t)$ are described by $\rho_{x(t)}(a)$. Recall that the transition between
different charts is given by conjugating. Therefore, the conjugacy
class $[\rho_{[x(t)]}(a)]$ is continuous on the
curve $[x(t)]$.

Since $s$ is semi-simple and $\rho:A\ra GL(V)$ is a group
homomorphism, the Jordan form of $\rho_{x(t)}(a)$ is a discrete subset in
the set of the conjugacy classes of $GL(V)$. Therefore,
$[\rho_{[x(t)]}(a)]$ is locally constant.
\end{proof}

We denote the collection of
the points $[B,\alpha,\beta]$ with the common conjugacy class $[\rho]$
by $\projQuot([\rho])$. Then $\projQuot([\rho])$ is a union of connected
components of $\projQuot(V,W)^A$. It follows that we have $\pi\projQuot(V,W)^A=\cup_{[\rho]}\pi\projQuot([\rho])=\sqcup_{[\rho]}\pi\projQuot([\rho])$. Denote
each stratum $\pi\projQuot([\rho])$ by $\affQuot([\rho])$.

Fix the conjugacy class $[\rho]$. Using the eigenvalues and
the eigenspaces of $s$ and $\rho(s,\varepsilon)$, we can endow $W$
and $V$ with gradings. Assume $W$
and $V$ have eigenspace decompositions $W=\oplus_i
W_i=\oplus_{i;a\in\Z}W_i(a)$,
$V=\oplus_iV_i=\oplus_{i;a\in\Z}V_i(a+\Hf)$, where $W_i(a)$ and
$V_i(a+\Hf)$ have eigenvalues $\varepsilon^{2(an+i-1)}$,
$\varepsilon^{2((a+\Hf)n+i-1)}$ respectively. Associate to
$[\rho(s,\varepsilon)]$ the bigraded vectors
$w=(\dim W_i(a))$, $v=(\dim V_i(a))$. We can identify
$\projQuot([\rho])$ with the graded quiver variety $\grProjQuot(v,w)$. Similarly, the graded categorical quotient
$\grAffQuot(v,w)$ is identified with the subvariety
$\affQuot([\rho])$ of $\affQuot(V,W)$.

\begin{Rem} 
If we take the $\C^*$-action in \cite{Nakajima09}, the representatives
of the fixed points in the sub-varieties
$\projQuot([\rho])$ do not take the form of
representations of the quiver in Figure \ref{fig:repSpace}. For example, let the quiver $Q$ be given by Figure
\ref{fig:acyclicQuiver}. For simplicity, let us assume $w_2=0$. Then these representatives are
representations of the quiver in Figure \ref{fig:mixedRepSpace}, where
the black arrows arise from those of $Q\op$, the green arrows
arise from those of $Q$, and the orange arrows correspond to the
linear maps $\alpha_i(a)$, $\beta_i(a)$, $i\in I$, $a\in \Z$.

Such representations do not suit our purpose.
\end{Rem}

\begin{figure}[htb!]
\centering
\begin{tikzpicture}
\node (v21) at (3,0) {$\cdots$};
\node (v24) at (13,0) {$\cdots$};
\node (v11) at (5,-2) {$V_1(1)$};
\node (v21) at (5,0) {$V_2(1)$};
\node (v31) at (5,2) {$V_3(1)$};
\node (v12) at (7,-2) {$V_1(2)$};
\node (v22) at (7,0) {$V_2(2)$};
\node (v32) at (7,2) {$V_3(2)$};
\node (v13) at (9,-2) {$V_1(3)$};
\node (v23) at (9,0) {$V_2(3)$};
\node (v33) at (9,2) {$V_3(3)$};
\node (v14) at (11,-2) {$V_1(4)$};
\node (v24) at (11,0) {$V_2(4)$};
\node (v34) at (11,2) {$V_3(4)$};

\node (w34) at (11,3) {$W_3(4)$};
\node (w14) at (11,-3) {$W_1(4)$};

\node (w33) at (9,3) {$W_3(3)$};
\node (w13) at (9,-3) {$W_1(3)$};

\node (w32) at (7,3) {$W_3(2)$};
\node (w12) at (7,-3) {$W_1(2)$};

\node (w31) at (5,3) {$W_3(1)$};
\node (w11) at (5,-3) {$W_1(1)$};

\tikzstyle{green_arrow_edge} = [->, draw=green!55, line width=1]

\draw[green_arrow_edge] (v24) edge (v33);
\draw[green_arrow_edge] (v14) edge (v33);
\draw[green_arrow_edge] (v14) edge (v23);

\draw[-triangle 60] (v34) edge (v23);
\draw[-triangle 60] (v34) edge (v13);
\draw[-triangle 60] (v24) edge (v13);
\draw[green_arrow_edge] (v23) edge (v32);
\draw[green_arrow_edge] (v13) edge (v32);
\draw[green_arrow_edge] (v13) edge (v22);

\draw[-triangle 60] (v33) edge (v22);
\draw[-triangle 60] (v33) edge (v12);
\draw[-triangle 60] (v23) edge (v12);
\draw[green_arrow_edge] (v22) edge (v31);
\draw[green_arrow_edge] (v12) edge (v31);
\draw[green_arrow_edge] (v12) edge (v21);

\draw[-triangle 60] (v32) edge (v21);
\draw[-triangle 60] (v32) edge (v11);
\draw[-triangle 60] (v22) edge (v11);

\tikzstyle{orange_edge} = [->, draw=orange!55, line width=1]

\draw[orange_edge] (v34) edge (w33);
\draw[orange_edge] (v33) edge (w32);
\draw[orange_edge] (v32) edge (w31);
\draw[orange_edge] (w34) edge (v33);
\draw[orange_edge] (w33) edge (v32);
\draw[orange_edge] (w32) edge (v31);

\draw[orange_edge] (v14) edge (w13);
\draw[orange_edge] (v13) edge (w12);
\draw[orange_edge] (v12) edge (w11);
\draw[orange_edge] (w14) edge (v13);
\draw[orange_edge] (w13) edge (v12);
\draw[orange_edge] (w12) edge (v11);
\end{tikzpicture}
\caption{The vector space $\grRep(Q\op,v,w)$ for $Q$ acyclic and the
  $\C^*$-action of \cite{Nakajima09}}
\label{fig:mixedRepSpace}
\end{figure}

The graded version of Proposition \ref{prop:smoothQuot} implies that the graded quiver variety $\grProjQuot(V,W)$ is smooth.

\begin{Prop}[{\cite[Corollary 5.5]{Nakajima94} \cite[Proposition 4.1.2]{Nakajima01}}]
$\projQuot([\rho])$ is homotopic to $\lag(V,W)\cap \projQuot([\rho])$.
\end{Prop}

\begin{proof}
By using Slodowy's technique \cite[Section 4.3]{Slodowy1980}, Nakajima has shown that $\projQuot(V,W)$ is homotopic to $\lag(V,W)$
 in \cite{Nakajima94}, and he has also proved this proposition with a different $GL(W)\times \C^*$-action in
\cite{Nakajima01}. The technique remains valid here. We shall briefly
recall it.

Equip the space $\rep(Q\op,V,W)$ with the $\C^*$-action \st we have $\varepsilon(B_h,B_{\oh},\alpha,\beta)=
(B_h,\varepsilon B_{\oh},\alpha,\varepsilon\beta)$ for any
$\varepsilon\in\C^*$. It commutes with the $GL(V)$-action. Furthermore,
the set of the stable points $\mu^{-1}(0)^s$ is invariant under this action. Therefore, we obtain a $\C^*$-action on $\projQuot(V,W)$. We have
\begin{align}
\lag(V,W)=\{[x]\in\projQuot(V,W)|\lim_{\varepsilon\ra \infty}\varepsilon[x]\ \mathrm{exists}\}.
\end{align}
Now the technique of Slodowy \cite[4.3]{Slodowy1980} implies that,
via this $\C^*$-action, $\projQuot(V,W)$ retracts to a
neighborhood of $\lag(V,W)$, \st $\lag(V,W)$ is a strong deformation
retract of this neighborhood. 

Because our $A$-action commutes with this $\C^*$-action, we can apply the above constructions to the $A$-fixed subsets $\grRep(Q\op,v,w)$, $\grProjQuot(v,w)$, $\grLag(v,w)$. Then the proposition is verified.
\end{proof}

Let us define $\grRegStratum(v,w)=\grRegStratum([\rho])=\pi(\pi^{-1}(\regStratum(V,W))\cap
\projQuot([\rho]))$. Then the morphism $\pi$ is an isomorphism from
$\pi^{-1}(\grRegStratum(v,w))$ to $\grRegStratum(v,w)$.

The maps in Theorem \ref{thm:transversalSlice} commute with the $GL(W)\times \C^*$ action. Restrict the maps to the subvarieties $\grProjQuot(v,w)$, $\grProjQuot(v^\bot ,w^\bot )$ of $\projQuot(V,W)$, $\projQuot(V^\bot ,W^\bot )$ respectively. We obtain a transversal slice theorem for graded quiver varieties.

For any $x\in\grRegStratum(v^0,w)\subset\grAffQuot(v,w)$, let $T$
denote the its tangent space in $\grRegStratum(v,w)$. Define $w^\bot$
and $v^\bot$ to be $w-C_qv^0$ and $v^\bot=v-v^0$ respectively.
\begin{Thm}[Transversal slice] \label{thm:grTransversalSlice}
There exist neighborhoods $\cU$, $\cU_\cT$, $\cU^\bot$ of $x$ in
$\grAffQuot(V,W)$ and the origins in $\cT$, $\grAffQuot(v^\bot,w^\bot )$ respectively, and biholomorphic maps $\cU\ra \cU_\cT\times \cU^\bot $, $\pi^{-1}(\cU)\ra \cU_\cT\times \pi^{-1}(\cU^\bot )$, \st the following diagram commutes:
\begin{equation*}
\begin{CD}
  \grProjQuot(v,w) \;@. \supset \; @. \pi^{-1}(\cU) @>>\cong> \cU_\cT \times \pi^{-1}(\cU^\bot ) \;@.\subset \;@. \cT\times \projQuot(v^\bot ,w^\bot )\\
   @. @. @V{\pi}VV @VV{\id\times\pi}V @.@. \\
   \grAffQuot(v,w)\;@.\supset \; @. \cU @>>\cong> \cU_\cT\times \cU^\bot  \;@. \subset \;@. \cT\times\grAffQuot(v^\bot ,w^\bot )
\end{CD}
\end{equation*}

Notice that the fibre $\pi^{-1}(x)$ is biholomorphic to the fibre
$\grLag(v^\bot ,w^\bot )$ over the origin.
\end{Thm}

\begin{Prop}
  \label{prop:grStratification}
The affine graded quiver variety $\grAffQuot(v,w)$ admits a stratification
\begin{align}
  \label{eq:grStratification}
 \sqcup_{(v',w)\geq (v,w)}\grRegStratum(v',w).
\end{align}
\end{Prop}
\begin{proof}
It suffices to show 
\begin{align}
\label{eq:stratification}
\pi(\projQuot(V,W)^A)=\sqcup_{V'}\pi(\pi^{-1}(\regStratum(V',W))\cap \projQuot(V,W)^A)
\end{align}
for each pair of vector spaces $(V,W)$. Then we obtain \eqref{eq:grStratification} by restricting \eqref{eq:stratification} to the
subvariety $\grAffQuot(v,w)=\affQuot([\rho])$ for the conjugacy class $[\rho]$
associated with $(v,w)$.

We claim that every point $[B,\alpha,\beta]$ of $\pi(\projQuot(V,W)^A)$ is contained in the
right hand side of \eqref{eq:stratification}. Using Proposition
\ref{prop:stratum} and Remark \ref{rem:stratum}, we see that
$[B,\alpha,\beta]$ belongs to some $\affQuot(V,W)_{(\hat{G})}$. If
$\hat{G}$ equals $\{e\}$, the claim is true. Otherwise, choose the
representative $(B,\alpha,\beta)$ in Proposition \ref{prop:stratum} and consider the $B$ actions on all the summand $V^s$,
$s>0$. If the actions are trivial, $[B,\alpha,\beta]$ lies
in the regular stratum $\regStratum(V^0,W)$, and the claim follows easily. If the action is
nontrivial for some $V^s$, $s>0$, we obtain a point other than $0$ in the categorical quotient
$\affQuot(V^s,0)$. Because the $GL(W)\times \C^*$-action is
compatible with the decomposition of $V$ in Proposition \ref{prop:stratum},
$\rho_{(B,\alpha,\beta)}$ stabilizes the decomposition. Let $v^s$ be the bigraded vector
associated with the $\rho_{(B,\alpha,\beta)}$-action on $V^s$. As in
Remark \ref{rem:stratum}, $V^s$ is a
representative of a nonzero point in $\grAffQuot(v^s,0)$. However in
our setting $\grAffQuot(v^s,0)$ is the categorical quotient of the
$v^s$-dimensional representations of some acyclic quiver, which is
always equal to $\{0\}$. This contradiction implies that the
$B$-action on $V^s$ must be trivial.
\end{proof}

\begin{Rem}
When $Q$ is of Dynkin
  type, the ungraded version of the proposition holds, \cf {\cite[Proposition 2.6.3]{Nakajima01}}. However, it does not necessarily hold
  when $Q$ is not of Dynkin type, \cf Example 10.10 in
  \cite{Nakajima98}. In general, whether the proposition is true or
  not depends on the choice of the $GL(W)\times \C^*$-action.
\end{Rem}

Let $m$ be any integer. Following \cite[Section 2.8]{Nakajima01}, we define a $\C^*$-module structure $L(m)$ on $\C$ by
\begin{align}
  \varepsilon\cdot v=\varepsilon^m v,
\end{align}
where $\varepsilon \in \C^*$, $v\in \C$.

For a $\C^*$-module M, we denote the $\C^*$-module $L(m)\otimes_{\C}M$ by $q^m M$.

As in [2.9, \cite{Nakajima01}], $\mu^{-1}(0)^s$ is a principal
$GL(V)$-bundle over $\projQuot(V,W)$. Therefore, for any $i\in I$, we can view the vector space
$V_i$ as an associated vector bundle by using the natural $GL(V_i)$
action. Also, it is naturally a $GL(W)\times \C^*$-equivariant
vector bundle \st $GL(W)$ acts trivially. Similarly, we view $W_i$ as a $GL(W)\times
\C^*$-equivariant vector bundle over $\projQuot(V,W)$ by using the
$GL(W_i)\times \C^*$ action \st $\C^*$ acts trivially. 

For each $k\in I$, we have the following complex $C^*_k=(\sigma_k,\tau_k)$ of $GL(W)\times
\C^*$-equivariant vector bundles:
\begin{align}\label{eq:dimComplex}
  C^*_k:q^{-2n}V_k\xra{\sigma_k}((\oplus_{j:i< k}q^{2(k-i-n)}V_i^{\oplus b_{ik}})\oplus(\oplus_{j:j> k}q^{2(j-k)}V_j^{\oplus b_{kj}})\oplus q^{-n}W_k)\xra{\tau_k}V_k,
\end{align}
where $\sigma_k=(\oplus_{h\in H, s(h)=k} B_{h})\oplus \beta_k$,
$\tau_k=\sum_{h\in \Omega, t(h)=k}B_h-\sum_{\oh\in
  \oOmega,t(\oh)=k}B_{\oh}+\alpha_k$ are $GL(W)\times
\C^*$-equivariant morphisms. This complex is just the complex in \cite[4.2]{Nakajima98} with a modification of $GL(W)\times \C^*$-action. Let the middle term be the degree $0$ component.

\begin{Prop}[{\cite[Lemma 2.9.2, Lemma 2.9.4]{Nakajima01}}]\label{prop:dominantComplex}
Fix a point $[B,\alpha,\beta]\in\projQuot(V,W)$ and consider $C_k^*$ as a complex of vector spaces, $k\in I$. 

1) The cohomology $H^{-1}(C_k^*)$ vanishes.

2) If the image $\pi([B,\alpha,\beta])\in \affQuot(V,W)$ is
contained in some regular stratum $\regStratum(V',W)\subset \affQuot(V,W)$, then $V'$
equals $V$ if and only if $H^1 (C_k^*)$ vanishes for all $k\in I$.
\end{Prop}

Fix a graded quiver variety $\grProjQuot(v,w)=\projQuot([\rho])$. Let $(C^*_k)^\bullet$ be the restriction of $C^*_k$ to $\grProjQuot(v,w)$. Then the complex $(C_k^*)^\bullet$ decomposes as
\begin{align}
  \label{eq:decomposeComplex}
  (C_k^*)^\bullet=\oplus_{b\in \Z+\Hf} C_k^*(b),
\end{align}
where the complexes $C_k^*(b)=(\sigma_k(b),\tau_k(b))$ are given by
\begin{align}\label{eq:grDimComplex}
  \begin{split}
    q^{-2n}V_k(b)\xra{\sigma_k(b)}((\oplus_{j:i<
      k}q^{2(k-i-n)}V_i(b)^{\oplus b_{ik}})&\oplus(\oplus_{j:j>
      k}q^{2(j-k)}V_j(b-1)^{\oplus b_{kj}}) \oplus
    q^{-n}W_k(b-\Hf))\\
&\xra{\tau_k(b)}V_k(b-1).
  \end{split}
\end{align}
The decomposition commutes with the $A$-action.

We have the following analogue of Proposition \ref{prop:dominantComplex} in graded cases.
\begin{Prop}
  \label{prop:grDominantComplex}
1) The cohomology $H^{-1}((C_k^*)^\bullet)$ vanishes.

2) The image $\pi([B,\alpha,\beta])\in \grAffQuot(v,w)$ is contained in the regular stratum $\grRegStratum(v,w)$ if and only if $H^1 ((C_k^*)^\bullet)$ vanishes, $\forall k\in I$.
\end{Prop}
\begin{proof}
Part 1) follows from Proposition \ref{prop:dominantComplex} by
restriction. For part 2), we additionally use Proposition \ref{prop:grStratification}.
\end{proof}

\begin{Thm}\label{thm:connected}
  $\grProjQuot(v,w)$ is connected.
\end{Thm}
\begin{proof}
The statement follows from the arguments in the proof of \cite[Theorem 5.5.6]{Nakajima01}. Notice
that, since our $GL(W)\times \C^*$-action is different from that of
\cite{Nakajima01}, we don't need the condition $|b_{ij}|\leq 1$, $1\leq i,j\leq n$, in \cite[Theorem 5.5.6]{Nakajima01}.
\end{proof}
As a consequence, the smooth variety $\grProjQuot(v,w)$ is irreducible.

\begin{Prop}[{\cite[Corollary 5.5.5]{Nakajima01}}]\label{prop:codim}
  On a nonempty open subset of $\grProjQuot(v,w)$, we have
  \begin{align}
    \label{eq:codim}
    \codim \Im\tau_{k}(b)=\max(0,-\rank (C_k^*(b))).
  \end{align}
\end{Prop}

\begin{Lem} \label{lem:rankDominantComplex}
$(\rank (C_k^*)^\bullet)_{k\in I}$ equals $w-C_qv$.
\end{Lem}

\begin{Rem}
  In their ongoing work \cite{KellerScherotzke2013}, Bernhard Keller and Sarah Scherotzke use
  representation theory to give explicit constructions of the points
  in $\grRegStratum(v,w)$, where $(v,w)$ is $l$-dominant. The special
  case where $Q$ is of Dynkin type has also been studied in \cite{LeclercPlamondon13}.
\end{Rem}

\begin{Prop}
  \label{prop:nonemptyStratum}
$\grRegStratum(v,w)$ is non-empty if and only if $(v,w)$ is $l$-dominant.
\end{Prop}
\begin{proof}
This follows from Proposition \ref{prop:grDominantComplex}
\ref{prop:codim} and Lemma \ref{lem:rankDominantComplex}.
\end{proof}
Because the restriction of $\pi$ over $\grRegStratum(v,w)$ is a local
homeomorphism, $\dim \grRegStratum(v,w)$ can be calculated by Lemma
\ref{lem:dForm}.

\begin{Rem}\label{rem:finiteDimensional}
For any given vector $w$, by induction on its width with respect to the lexicographical order, one can prove that there are only finitely many $v$ \st $(v,w)$ is $l-dominant$. Then Propositions \ref{prop:nonemptyStratum} and \ref{prop:grStratification} imply that $\grAffQuot(w)=\cup_{v}\grAffQuot(v,w)$ is finite-dimensional.
\end{Rem}

\begin{Cor}\label{cor:degeneratedStratum}
For any $l$-dominant pairs $(v,w)$, $(v^0,w)$, $\grRegStratum(v^0,w)$
is contained in the closure of $\overline{\grRegStratum(v,w)}$ if and
only if $(v^0,w)$ dominates $(v,w)$.
\end{Cor}
\begin{proof}
Proposition \ref{prop:nonemptyStratum} implies that both regular
strata are non-empty. By Theorem \ref{thm:connected},
$\grAffQuot(v,w)$ is irreducible. Since $\grRegStratum(v,w)$ is a non-empty open subset in $\grAffQuot(v,w)$, its closure equals $\grAffQuot(v,w)$. 

If $(v^0,w)\geq (v,w)$, $\grAffQuot(v^0,w)$ is naturally embedded into $\grAffQuot(v,w)$. Thus the only if part holds.

Conversely, if $\grRegStratum(v^0,w)\subset \overline{\grRegStratum(v,w)}=\grAffQuot(v,w)$, using proposition \ref{prop:grStratification} we obtain the if part.
\end{proof}

Following \cite[(4.9)]{Nakajima04}, for any given pairs of vectors
$(v^1,w^1)$ and $(v^2,w^2)$, we define the following complex of vector
bundles over
$\grProjQuot(v^1,w^1)\times \grProjQuot(v^2,w^2)$:
\begin{align}
  \label{eq:vectorBundleComplex}
  \grEndSp(v^1,v^2)\xra{\sigma^{21}} \opname{E}(v^1,v^2)\oplus  \opname{L}(w^1,v^2)\oplus \opname{L}(v^1,w^2)\xra{\tau^{21}} \grEndSp(v^1,v^2[-1]),
\end{align}
where the middle term has degree $0$, and we denote
\begin{align}
  \label{eq:sigmaTau}
\sigma^{21}(\xi)=(B^2\xi-\xi B^1)\oplus (-\xi \alpha^1) \oplus \beta^2 \xi,\\  
\tau^{21}(C\oplus I\oplus J)=\epsilon B^2 C+\epsilon C B^1+\alpha^2 J+I \beta^1.
\end{align}
  The complex is exact on the left and on the right, \cf the argument
  in \cite[3.10]{Nakajima98}.

\begin{Lem}
    The quotient $\ker \tau^{21}/\Im \sigma^{21}$ is a
    vector bundle over $\grProjQuot(V^1,W^1)\times
    \grProjQuot(V^2,W^2)$.
\end{Lem}
\begin{proof}
By
\cite{CrawleyBoeveyJensen05}, any subvariety of a vector bundle over
the smooth variety $\grProjQuot(v^1,w^1)\times \grProjQuot(v^2,w^2)$, whose intersection with every fibre is a vector subspace of constant
dimension, is a sub-bundle. Therefore, $\ker \tau^{21}$ is a vector bundle. Let $(\sigma^{21})^*$ be the transpose
(or dual) of $\sigma^{21}$ restricted to $\ker \tau^{21}$. Its kernel $(\ker \tau^{21}/\Im \sigma^{21})^*$ is a sub-bundle of $(\ker
\tau^{21})^*$. Therefore, $\ker \tau^{21}/\Im \sigma^{21}$ is
again a vector bundle.
\end{proof}

Denote the rank of the vector bundle $\ker \tau^{21}/\Im
\sigma^{21}$ by $d((v^1,w^1),(v^2,w^2))$.
\begin{Lem}
The rank $d((v^1,w^1),(v^2,w^2))$ is given by
\begin{align}
  \label{eq:dimension}
(w^1-C_qv^1)\cdot v^2[-\Hf]+ v^1\cdot w^2[-\Hf].  
\end{align}
\end{Lem}
\begin{proof}
 It suffices to calculate the rank of the complex. We have
  \begin{align*}
    d((v^1,w^1),(v^2,w^2))&=\sum_{k\in I,a\in \Z} v_k^2(a-\Hf)(\sum_{i:i<k}b_{ik}v_i^1(a+\Hf)+\sum_{j:j>k}b_{kj}v_j^1(a-\Hf))\\
&\qquad +\sum_{k\in I,a\in \Z}(w_k^1(a)v_k^2(a-\Hf)+v_k^1(a-\Hf)w_k^2(a-1))\\
&\qquad - \sum_{k\in I,a\in\Z}v_k^2(a-\Hf)(v_k^1(a-\Hf)+v_k^1(a+\Hf))\\
&=\sum_{k,a}(v_k^2(a-\Hf)(-C_qv^1)_k(a)+w_k^1(a)v_k^2(a-\Hf)+w_k^2(a)v_k^1(a+\Hf))\\
&=-v^2[-\Hf]\cdot C_qv^1+w^1\cdot v^2[-\Hf]+w^2 \cdot v^1[\Hf]\\
&=\RHS.
  \end{align*}
\end{proof}

\begin{Rem}
  Using equation \eqref{eq:symmetricForm}, we can easily rewrite the
  quadratic form \eqref{eq:dimension} as
  \begin{align}
     d((v^1,w^1),(v^2,w^2))=v^1[\Hf]\cdot (w^2-C_qv^2)+ w^1[\Hf]\cdot v^2.  
  \end{align}
The reader is invited to compare this expression with \cite[(2.1)]{Nakajima04}.
\end{Rem}

Let $[B,\alpha,\beta]$ be any point in
$\grProjQuot(v,w)=\projQuot([\rho])\subset \projQuot(V,W)$. The
tangent space of $\grProjQuot(v,w)$ at $[B,\alpha,\beta]$ is the
$A$-fixed part of the tangent space of $\projQuot(V,W)$ at the same
point. The $A$-fixed part of complex \eqref{eq:tangentComplex} over
this point is just
\begin{align}
  \label{eq:grTangentComplex}
  \opname{L}(v,v[-1])\xra{\iota} \opname{E}(v,v)\oplus \opname{L}(v,w)\oplus \opname{L}(w,v)\xra{d\mu} \opname{L}(v,v[1]).
\end{align}
Comparing it with the complex \eqref{eq:vectorBundleComplex}, we have
the following result.

\begin{Lem}\label{lem:dForm}
The dimension of $ \grProjQuot(v,w)$ equals $d((v,w),(v,w))$.
\end{Lem}

\begin{Thm}[{\cite[Theorem 7.4.1]{Nakajima01}}]\label{thm:oddVanish}
The odd homology of $\grLag(v,w)$ vanishes. And we have a
nondegenerate pairing
between $H_*(\grLag(v,w))$ and $H_*(\grProjQuot(v,w))$.
\end{Thm}
\begin{proof}
We refer the reader to \cite[Section 7]{Nakajima01} for the proof. This theorem is just a consequence of \cite[Theorem 7.4.1]{Nakajima01}.
\end{proof}

\section{$qt$-characters}
\label{sec:qtCharacters}

In the study of finite-dimensional representations
of quantum affine algebras, Nakajima invented the $t$-analogues of the $q$-characters
($qt$-characters for short) \cite{Nakajima04}, which are defined as the generating series of the Betti numbers of the graded quiver varieties. In this section, we
generalize his constructions for our graded quiver
varieties and introduce a new $t$-deformation.

\subsection{Quadratic forms}
\label{sec:quadraticForm}
For any pairs $(v,w)$, $(v',w')$, recall that we
have the quadratic form $d((v,w),(v',w'))$ given by $\eqref{eq:dimension}$, whose geometric meaning
is explained by Lemma \ref{lem:dForm}. Let us
further define the quadratic forms $\dT((v,w),(v',w'))$, $\dTW(w,w')$
and $\eMatrix(w,w')$ \st
\begin{align*}
\dTW(w,w')=-w[\Hf] \cdot C_q^{-1}w',\\
  \eMatrix(w,w')=-w[\Hf] \cdot C_q^{-1}w'+w'[\Hf] \cdot C_q^{-1}w,\\
  \dT((v,w),(v',w'))=d((v,w),(v',w'))+\dTW(w,w').
\end{align*}
Our definitions are slightly different from those of \cite{Nakajima04}. It follows from our definitions that $\dT((0,w),(0,w'))=\dTW(w,w')$.

\begin{Lem}\label{eq:chopTwistDimension}
The form $\dT((v,w),(v',w'))$ equals $\dT((0,w-C_q v),(0,w'-C_q v'))$.\end{Lem}
\begin{proof}
We have
\begin{align*}
    \dT((v,w),(v',w'))&=(w-C_qv)\cdot v'[-\Hf]+v\cdot w'[-\Hf]-w[\Hf]\cdot C_q^{-1}w'\\
    &=(w-C_qv)\cdot v'[-\Hf]+v\cdot C_q C_q^{-1}w'[-\Hf]-w[\Hf]\cdot C_q^{-1}w'\\
    &\stackrel{\eqref{eq:symmetricForm}}{=}(w-C_qv)\cdot v'[-\Hf]+C_qv\cdot C_q^{-1}w'[-\Hf]-w[\Hf]\cdot C_q^{-1}w'\\
    &=-(w-C_qv)[\Hf]\cdot C_q^{-1}(w'-C_q v')\\
&=\dT((0,w-C_q v),(0,w'-C_q v')).
\end{align*}
\end{proof}

Define $e_{k,a}$ to be the unit $I\times \Z$-graded vector concentrated at
the degree $(k,a)$. Define $p_{i,j}=\dim\Hom_{\mod \C Q}(P_i,P_j)$,
where $P_i$, $i\in I$, is the $i$-th projective right module of $\C Q$.

\begin{Lem}\label{lem:epsilon}
1) For any $a\in \Z$ and $i,j\in I$, we have
\begin{align}
  \label{eq:epsilonValue}
 \eMatrix(e_{i,a},e_{j,a})=0,\\
\eMatrix(e_{i,a},e_{j,a-1})=-p_{ji},\\
\eMatrix(e_{i,a-1}+e_{i,a},e_{j,a-1}+e_{j,a})=p_{ij}-p_{ji}.
\end{align}

2) Let $B_{\tQ}$ be the $2n\times 2n$-matrix associated with $\tQ^z_1$ the
level $1$ ice quiver with $z$-pattern (\cf Figure \ref{fig:levelOneQuiver}
for an example and \cite{QinThesis} for the definition) whose principal part is
$Q$. Define $\zLambda$ to be the $2n\times 2n$ matrix \st for $1\leq i,j\leq n$, we have
\begin{align}
  \label{eq:lambdaMatrix}
  \zLambda_{i,j}= -\eMatrix(e_{i,0},e_{j,0})=0,\\
\zLambda_{i+n,j}=-\eMatrix(e_{i,-1}+e_{i,0},e_{j,0})=-p_{ij},\\
\zLambda_{i,j+n}=-\eMatrix(e_{i,0},e_{j,-1}+e_{j,0})=p_{ji},\\
\zLambda_{i+n,j+n}=-\eMatrix(e_{i,-1}+e_{i,0},e_{j,-1}+e_{j,0})=-p_{ij}+p_{ji}.
\end{align}
Then $\zLambda$ is the inverse of $-B_{\tQ}$.
\end{Lem}
\begin{proof}
1) The claim follows from straightforward computation.

2) Apply the results of 1) for $a=0$. The entries of $B_{\tQ}$ are
given by
\begin{align*}
  (B_{\tQ})_{i,j}=b_{ij},\\
  (B_{\tQ})_{i+n,j}=-\delta_{ij}+[b_{ji}]_{+},\\
  (B_{\tQ})_{i,j+n}=-(B_{\tQ})_{j+n,i},\\
  (B_{\tQ})_{i+n,j+n}=0,
\end{align*}
for any $1\leq i,j\leq n$. Here $[\ ]_{+}$ is defined to be $max\{\
,0\}$.

We have, for any $1\leq i,k\leq n$,
\begin{align*}
  (-\zLambda B_{\tQ})_{i,k}&=\sum_{1\leq j\leq n}\zLambda_{i,j+n}\cdot(-B_{\tQ})_{j+n,k}\\
&=\sum_{1\leq j\leq
    n}p_{ji}\delta_{jk}-\sum_{1\leq j\leq n}p_{ji}[b_{kj}]_{+}\\
  &=p_{ki}-\sum_j[b_{kj}]_{+}p_{ji}\\
  &=\delta_{ik},
\end{align*}
\begin{align*}
  (-\zLambda B_{\tQ})_{i+n,k}&=\sum_{1\leq j\leq
    n}(-\zLambda_{i+n,j})\cdot (B_{\tQ})_{j,k}+\sum_{1\leq j\leq
    n}(-\zLambda_{i+n,j+n})\cdot (B_{\tQ})_{j+n,k}\\
&=\sum_{1\leq j\leq
    n}p_{ij}b_{jk}+\sum_{1\leq j\leq n}(p_{ij}-p_{ji})(-\delta_{jk}+[b_{kj}]_{+})\\
  &=\sum_{j:j>k}p_{ij}(-b_{kj})+\sum_{j:j<k}p_{ij}b_{jk}+(-p_{ik}+p_{ki})\\
&\qquad +\sum_j(p_{ij}-p_{ji})[b_{kj}]_{+}\\
  &=(\sum_{j:j>k}p_{ij}(-b_{kj})+\sum_jp_{ij}[b_{kj}]_{+})+(\sum_{j:j<k}p_{ij}b_{jk}-p_{ik})\\
&\qquad +(p_{ki}-\sum_j[b_{kj}]_{+}p_{ji})\\
  &=0-\delta_{ik}+\delta_{ik}\\
  &=0,\\
  (-\zLambda B_{\tQ})_{i,k+n}&=(\zLambda B_{\tQ})_{k+n,i}=0,\\
  (-\zLambda B_{\tQ})_{i+n,k+n}&=\sum_{1\leq j\leq n}p_{ij}(\delta_{kj}-[b_{jk}]_{+})\\
  &=p_{ik}-\sum_{j}p_{ij}[b_{jk}]_{+}\\
  &=\delta_{ik}.
\end{align*}
\end{proof}

\subsection{Deformed Grothendieck ring}
\label{sec:perverseSheaves}

Recall that, by Proposition \ref{prop:grStratification}, we have the stratification
\begin{align*}
  \grAffQuot(w)=\cup_v \grAffQuot(v,w)=\sqcup_{v:(v,w)\text{ is $l$-dominant}} \grRegStratum(v,w).
\end{align*}

Let
$IC(v,w)$ denote the intersection cohomology sheaf associated with
$\grRegStratum(v,w)$ for $l$-dominant pair $(v,w)$ and
$1_{\grProjQuot(v,w)}$ the perverse sheaf
$\underline{\C}[\dim\grProjQuot(v,w)]$. Applying the decomposition theorem
\cite{BeilinsonBernsteinDeligne82} to the sheaf
$\pi(v,w)=\pi_!(1_{\grProjQuot(v,w)})$, we obtain the following result.
\begin{Thm}\label{thm:decomposition}
We have a decomposition
  \begin{align}
    \label{eq:decomposition}
    \pi(v,w)=\oplus_{v':(v',w)\text{ is
        $l$-dominant},d\in\Z}IC(v',w)[d]^{\oplus a_{v,v';w}^d}.
  \end{align}
Moreover, we have $a_{v,v';w}^d\in\Z_{\geq 0}$,
$a_{v,v';w}^d=a_{v,v';w}^{-d}$, $a_{v,v;w}^d=\delta_{d0}$ for
$l$-dominant $(v,w)$, and
$a_{v,v';w}^d\neq 0$ only if $(v,w)\leq(v',w)$.
\end{Thm}
\begin{proof}
 By Proposition \ref{prop:grStratification}, Proposition
 \ref{prop:nonemptyStratum}, Corollary \ref{cor:degeneratedStratum},
 the summands appearing in the decomposition of
 $\pi_!(1_{\grProjQuot(v,w)})$ are shifts of simple perverse sheaves
 generated by local systems over $\grRegStratum(v',w)$ for
 $l$-dominant pairs $(v',w)$.

The argument in the proof of \cite[Theorem 14.3.2]{Nakajima01} implies the following results:
 \begin{enumerate}
 \item $a_{v,v;w}^d=\delta_{d0}$ if $(v,w)$ is $l$-dominant, since $\pi$ is an
   isomorphism between $\pi^{-1}\grRegStratum(v,w)\subset \grProjQuot(v,w)$ and $\grRegStratum(v,w)$;

\item Because we have Theorem \ref{thm:grTransversalSlice}, the
  transversal slice technique in the proof of \cite[Theorem
  14.3.2]{Nakajima01} is effective. It follows that the decomposition
  summands are shifts of simple perverse sheaves $IC(v',w)$ generated by trivial local systems.
 \end{enumerate}

Finally, since $\id_{\grProjQuot(v,w)}$ is invariant under the Verdier duality $D$, the sheaf $\pi(v,w)$ is invariant as well, and therefore we have $a_{v,v';w}^d=a_{v,v';w}^{-d}$.
\end{proof}

We construct a positive Laurent polynomial from the coefficients $a_{v,v';w}$: 
\begin{align*}
  a_{v,v';w}(t)=\sum_{d\in\Z}a_{v,v';w}^d t^d\in\N[t^\pm].
\end{align*}
Theorem \ref{thm:grTransversalSlice} implies that we have
$a_{v,v';w}(t)=a_{v-v',0;w-C_q v}(t)$.

For any given $w$, we have a finite set by Remark \ref{rem:finiteDimensional}: 
\begin{align*}
  \cP_w=\{IC(v,w)|\text{the pair}\ (v,w)\text{ is $l$-dominant, }\}.
\end{align*}

As usual, we consider the bounded derived category
$\cD_c(\grAffQuot(w))$ of
constructible sheaves on $\grAffQuot(w)$ and its full subcategory
$\cQ_w$ whose objects are the direct sums of $IC(v,w)[d]$,
$IC(v,w)\in\cP_w$, $d\in\Z$. Let $\KGp_w$ be the Grothendieck group of
$\cQ_w$. It is naturally a free $\tBase$-module whose module structure
is given by $t(L)=(L[1])$, $L\in\KGp_w$. The Verdier duality $D$
induces an involution $\overline{(\ )}$ on $\KGp_w$.

The set $\{IC(v,w)\}$ is a $\tBase$-basis of the Grothendieck group
$\KGp_w$. Notice that the sheaves $IC(v,w)$ are defined only for
the $l$-dominant pairs $(v,w)$. By
Theorem \ref{thm:decomposition}, $\KGp_w$ has another $\tBase$-basis $\{\pi(v,w)|(v,w)\text{ is
  $l$-dominant}\}$. It follows that the group
$\dualKGp_w=\Hom_\tBase(\KGp_w,\tBase)$ has two dual bases
$\{\can(v,w)\}$ and $\{\chi(v,w)\}$ respectively. Define the third basis
$$\{\pbw(v,w)|(v,w)\text{ is $l$-dominant}\}$$ whose elements are
given by
\begin{align*}
\langle \pbw(v,w),L\rangle=\sum_k t^{\dim \grRegStratum(v,w)-k}\dim H^k(i^!_{x_{v,w}}L),
\end{align*}
where $x_{v,w}$ is any regular point in $\grRegStratum(v,w)$, $i_{x_{v,w}}$
the inclusion, $L\in\cQ_w$, and $\langle\ ,\ \rangle$ the natural pairing. Notice
that, by Theorem \ref{thm:grTransversalSlice}, $\pbw(v,w)$ does not
depend on the choice of $x_{v,w}$. Furthermore, we have
\begin{align*}
  \langle \pbw(v',w),IC(v,w)\rangle=\langle
  \pbw({v'}^\bot,w^\bot),IC(v^\bot,w^\bot)\rangle.
\end{align*}
The properties of perverse sheaves imply 
\begin{align}
  \label{eq:canToPbw}
  \can(v,w)\in\pbw(v,w)+\sum_{(v',w)<(v,w)}t^{-1}\Z[t^{-1}]\pbw(v',w).
\end{align}
In particular, $\{\pbw(v,w)\}$ is a $\tBase$-basis.

Following \cite[Section 8]{Nakajima04},
let us denote the inverse basis transform by
\begin{align}\label{eq:pbwToCan}
  \pbw(v,w)=\can(v,w)+\sum_{(v',w)<(v,w)}Z_{v,v';w}(t)\can(v',w),
\end{align}
where $Z_{v,v';w}(t)\in t^{-1}\Z[t^{-1}]$.

As in \cite[3.3]{Nakajima09}, we define 
\begin{align}
  \quotKGp=\{f=(f_w)\in\prod_w\dualKGp_w|\langle
    f_w,IC(v,w)\rangle=\langle
    f_{w^\bot},IC(v^\bot,w^\bot)\rangle,\forall IC(v,w)\in\cP_w\}.
\end{align}
It is a free $\tBase$-module. Define $M(w)=(f_{w'})\in \prod_{w'}\dualKGp_{w'}$ to be the element in
$\quotKGp$ such that $f_w=M(0,w)$ and
$L(w)=(f_{w'})$ the element in $\quotKGp$ such that $f_w=L(0,w)$. Then $\{M(w)\}$ and
$\{L(w)\}$ are $\tBase$-bases of $\quotKGp$.

Next, we endow $\quotKGp$ with an associative multiplication $\otimes$. By \cite{VaragnoloVasserot03} (\cf also \cite[Section
3.5]{Nakajima09}), there is a restriction functor for any given
decomposition $w=w^1+w^2$:
\begin{align*}
  \tRes_{w^1,w^2}^w:\cD_c(\grAffQuot(w))\ra
\cD_c(\grAffQuot(w^1)\times \grAffQuot(w^2)).
\end{align*}
In particular, $\tRes^w_{w^1,w^2}(\pi(v,w))$ equals 
\begin{align*}
\oplus_{v^1+v^2=v}\pi(v^1,w^1)\boxtimes  \pi(v^2,w^2)[d((v^2,w^2),(v^1,w^1))-d((v^1,w^1),(v^2,w^2))].
\end{align*}

The shifted functor $\res^w_{w^1,w^2}[-\eMatrix(w^1,w^2)]$ induces a homomorphism from $\KGp_w$ to
$\KGp_{w^1}\otimes_{\tBase} \KGp_{w^2}$, which we also denote by
$\res^w_{w^1,w^2}[-\eMatrix(w^1,w^2)]$. For any given $w$, define the
homomorphism $\res^w$ from $\KGp_w$ to
$\oplus_{w^1+w^2=w}(\KGp_{w^1}\otimes_{\tBase} \KGp_{w^2})$ to be the direct sum
$\sum_{w^1+w^2=w}\tRes^w_{w^1,w^2}[-\eMatrix(w^1,w^2)]$. By Theorem
\ref{thm:grTransversalSlice}, the homomorphisms $\res^w$ induce a
multiplication of $\quotKGp$, which we denote by $\otimes$.
\begin{Thm}
The structure constants of the
  basis $\{\can(w)\}$ are positive:
  \begin{align}
    \label{eq:otimes}
    \can(w^1)\otimes
    \can(w^2)=\sum_{w^3}\canStr^{w^3}_{w^1,w^2}(t)\can(w^3),\ \canStr^{w^3}_{w^1,w^2}(t)\in\N[t^\pm].
  \end{align}
\end{Thm}
\begin{proof}
  The statement follows from the argument of \cite[Section 5.1 Lemma 1(b)]{VaragnoloVasserot03}.
\end{proof}

\subsection{Rings of characters}
\label{sec:targSpace}
Define the ring
  \begin{align}
    \targSpace=\tBase[[W_i(a),V_i(b)]]_{i\in \vtx, a\in \Z,
    b\in (\Z+\Hf)},
  \end{align}
  where $t$, $W_i(a)$, $V_i(b)$ are indeterminates. We denote its product
  by $\cdot$, and often omit this notation.

For any $(v,w)$, we let $m(v,w)$ denote the
monomial $$m=m(v,w)=\prod_{i,a}
W_i(a)^{w_i(a)}\prod_{i,b}V_i(b)^{v_i(b)}.$$
Endow $\targSpace$ with the twisted product $*$ and the bar involution
$\overline{(\ )}$ such that we have
\begin{align}
\overline{t}=t^{-1},\ \overline{m}=m,\\
  m^1*m^2=t^{-\dT(m^1,m^2)+\dT(m^2,m^1)}m^1m^2,
\end{align}
where $m$, $m^1$, $m^2$ are any monomials, and we define
\begin{align*}
  \dT(m^1(v^1,w^1),m^2(v^2,w^2))=\dT((v^1,w^1),(v^2,w^2)).
\end{align*}

We follow the notation of Section
\ref{sec:perverseSheaves}. Moreover, as in Theorem
\ref{thm:grTransversalSlice}, we write $\grFib_x(v,w)$ for the fibre
$\pi^{-1}(x)$ of a point $x$ under $\pi:\grProjQuot(v,w)\ra \grAffQuot(v,w)$. Notice
that the homology of the fibre $H_*(\grFib_{x_{v,w}}(v',w),\C)$ is isomorphic to
 $H^{\dim\grProjQuot(v',w)-*}(i_{x_{v,w}}^!\pi(v',w))$, \cf \cite[8.5.4]{ChrissGinzburg97}. Therefore, Theorem
\ref{thm:grTransversalSlice} implies that
\begin{align*}
  \langle \pbw(v,w),\pi(v',w) \rangle=\langle
  \pbw(v^\bot,w^\bot),\pi(v'-v^0,w^\bot) \rangle.
\end{align*}

\begin{Def}[\emph{$t$-analogue of $q$-characters}]For any given $w$, define
  $\qtChar(\ )$ to be
 the $\tBase$-linear map from
  $\dualKGp_w$ to $\targSpace$ \st
  \begin{align}
    \label{eq:qtChar}
    \qtChar(\ )=\sum_v \langle \ ,\pi(v,w)\rangle
W^wV^v.
  \end{align}
\end{Def}

We can compute
\begin{align*}
  \qtChar(L(v,w))&=\sum_{v'}a_{v',v;w}(t)W^wV^{v'}\\
&=\sum_{v'-v}a_{v'-v,0;w}(t)W^wV^{v'}.
\end{align*}
Since $W^wV^{v'}$ and $a_{v',v;w}(t)$ are
bar--invariant, $\qtChar(\ )$ is bar-invariant as well. Similarly, we
define $\pbwTarg(v,w)$ to be
\begin{align*}\tqChar(\pbw(v,w))
    &=\sum_{v',k}t^{\dim\grProjQuot(v,w)-k}\dim H^k(i_{x_{v,w}}^!\pi(v',w),\C)W^wV^{v'}\\
    &=\sum_{v',k}t^{\dim\grProjQuot(v,w)-\dim\grProjQuot(v',w)}t^{\dim\grProjQuot(v',w)-k}\dim
    H^k(i_{x_{v,w}}^!\pi(v',w),\C)W^wV^{v'}\\
    &=\sum_{v',l}t^{\dim\grProjQuot(v,w)-\dim\grProjQuot(v',w)}t^{l}\dim
    H_{l}(\grFib_{x_{v,w}}(v',w),\C)W^wV^{v'}\\
    &\stackrel{\ref{thm:oddVanish}}{=}\sum_{v'}t^{\dim\grProjQuot(v,w)-\dim\grProjQuot(v',w)}P_t(\grFib_{x_{v,w}}(v',w))W^wV^{v'}\\
    &=\sum_{v'-v}t^{-\dim\grProjQuot(v'-v,w-C_q v)}P_t(\grLag(v'-v,w-C_qv))W^wV^{v'}.
\end{align*}
Here $P_t(\ )$ is the twisted Poincar\'{e} polynomial $\sum_k (-t)^k \dim
  H_k(\ )$.

\begin{Def}[dual PBW elements]\label{def:pbwTarg}
  For any given $w$, define $\pbwTarg(w)$ to be the generating series 
  \begin{align}
    \label{eq:pbwTarg}
    \pbwTarg(w)=\qtChar(\pbw(0,w))=\sum_{v}P_t(\grLag(v,w))t^{-d((v,w),(v,w))}W^wV^v.
  \end{align}
\end{Def}
 Notice that the summation runs over all $v$ and $(v,w)$ is not
 necessarily $l$-dominant.

\begin{Rem}
  By Theorem \ref{thm:oddVanish}, the parameter $-t$ in $P_t(\ )$ in Definition \ref{def:pbwTarg}
  can be replaced by $t$.
\end{Rem}

Let $E(\ ;x,y)\in\Z[x,y]$ be the virtual Hodge Polynomial, \cf \eg \cite{HauselVillegas08}. Define
$p_t(\ )$ to be the virtual Poincar\'{e} polynomial (called the $E$-polynomial in
\cite{Qin10}), \ie $p_t(\grLag(v,w))=E(\grLag(v,w);t,t)$. It is
additive with respect to $\alpha$-partitions in the sense of \cite[Section 7.1]{Nakajima01}, and multiplicative with
respect to fibre products.

\begin{Prop}[{\cite[Lemma 5.2]{Nakajima04}}]
We have $p_t(\grLag(v,w))=P_t(\grLag(v,w))$ and also $p_t(\grProjQuot(v,w))=P_t(\grProjQuot(v,w))$.
\end{Prop}

Recall that, in $\dualKGp_w$, we have
$M(v,w)=\sum_{v'}Z_{(v,w)(v',w)}(t)L(v',w)$, where the matrix of coefficients
$(Z_{(v,w)(v',w)})$ is an upper uni-triangular matrix with respect to
the dominance order. Therefore, we can solve the following equation recursively
\begin{align*}
  \overline{M(v,w)}=\sum_{v':(v',w)\leq (v,w)} u_{(v,w),(v',w)}M(v',w),
\end{align*}
where the coefficients $u_{(v,w),(v',w)}\in \tBase$.

Finally, we apply $\qtChar$ to the above equation and obtain
  \begin{align}
    \label{eq:bar_pbwTarg}
    \overline{\pbwTarg(v,w)}=\sum_{v':(v',w)\leq (v,w)} u_{(v,w),(v',w)}\pbwTarg(v',w),
  \end{align}
\begin{Rem}
  In fact, we can compute the coefficient matrix
  $(u_{(v,w)(v',w)})$ and the basis transition matrix $(Z_{(v,w)(v',w)})$
   from $P_t(\grLag(v,w))$, \cf \cite[Section
  8]{Nakajima04}.
\end{Rem}

 \begin{Prop} [{\cite[Proposition 6.2]{Nakajima04}}]\label{prop:mult_pbwTarg}
If for all $i$, $j\in I$, $a$, $b\in\Z$, $a>b$, either $ (w^1)_i(a)$ or
$(w^2)_j(b)$ vanishes, 
then we have
\begin{align}
  \label{eq:mult_pbwTarg}
\pbwTarg(w^2)*\pbwTarg(w^1)= t^{\eMatrix(w^1,w^2)}\pbwTarg(w^1+w^2). 
\end{align}
\end{Prop}

\begin{proof}
As in the proof of \cite[Proposition 6.2]{Nakajima04}, for any
monomials $m^1=W^{w^1}V^{v^1}$, $m^2=W^{w^2}V^{v^2}$, we have a
complex vector bundle $\grVb(v^1,w^1;v^2,w^2)$ of rank given by 
$\Ker \tau^{21}/\Im \sigma^{21}=d((v^1,w^1),(v^2,w^2))$ over each 
$\grLag(v^1,w^1)\times \grLag(v^2,w^2)$. The set of $\grVb(v^1,w^1;v^2,w^2)$ for various $v^1$,
$v^2$, gives a stratification of
$\grLag(v^1+v^2,w^1+w^2)$, \cf \cite[6.12]{Nakajima01b}. Using the additivity and
multiplicativity of $p_t(\ )$, we deduce
\begin{align*}
&  \pbwTarg(w^2)*\pbwTarg(w^1)\\
&=\sum_{v^1,v^2}p_t(\grLag(v^1,w^1))\times
  p_t(\grLag(v^2,w^2))t^{-d(m^1,m^1)-d(m^2,m^2)}m^2*m^1\\
&=\sum_{v^1,v^2}p_t(\grLag(v^1,w^1)\times
  \grLag(v^2,w^2))t^{-d(m^1,m^1)-d(m^2,m^2)}\\&\qquad 
t^{\dT(m^1,m^2)-\dT(m^2,m^1)}m^1m^2\\
&=\sum_{v^1,v^2}p_t(\grVb(v^1,w^1;v^2,w^2))t^{-2d(m^1,m^2)}t^{-d(m^1,m^1)-d(m^2,m^2)}\\&\qquad 
t^{\eMatrix(w^1,w^2)+d(m^1,m^2)-d(m^2,m^1)}m^1m^2\\
&=\sum_{v^1,v^2}p_t(\grVb(v^1,w^1;v^2,w^2))t^{\eMatrix(w^1,w^2)}t^{-d(m^1m^2,m^1m^2)}m^1m^2\\
&=\RHS.
\end{align*}
\end{proof}

In order to study quantum cluster algebras, we
also need the following ring of formal power series
  \begin{align}
    \redTargSpace=\tBase[[Y_i(a)^\pm]]_{i\in \vtx, a\in \Z},
  \end{align}
  where $t$, $Y_i(a)$ are indeterminates. We often omit its usual
  product $\cdot$. We associate the monomial $m(v,w)=Y^{w-C_q v}$ to
  any given pair $(v,w)$. Then similar to the case of $\targSpace$, we endow $\redTargSpace$ with a bar involution $\bar{(\ )}$ and a twisted product $*$ arising
  from the bilinear form $\dT$.

Define the $\tBase$-linear map $\contr$ from $\targSpace$ to
$\redTargSpace$ \st for any $(v,w)$, it sends any monomial
$m(v,w)\in\targSpace$ to the monomial $m(v,w)=Y^{w-C_q
  v}\in\redTargSpace$. Then it is a ring homomorphism. Define the $t$-analogue of the
$q$-character map to be the $\tBase$-linear map $\tChar(\ )$ from
  $\quotKGp$ to $\redTargSpace$ given by
  \begin{align}
    \label{eq:redQtChar}
    \tChar(\ )=\sum_v \langle\  ,\pi(v,w)\rangle Y^{w-C_qv}.
  \end{align}

\begin{Thm}\label{thm:injectiveHom}
1) $\qtChar(\ )$ is injective.

2) $\tChar(\ )$ is an injective algebra homomorphism from $\quotKGp$ to $\redTargSpace$.
\end{Thm}
\begin{proof}
  The statements follow from the same arguments in\cite{VaragnoloVasserot03} and \cite[Theorem 3.5]{Nakajima04}.
\end{proof}

\section{Results on generic characters}\label{sec:pseudoModules}

From this section, let us restrict to level $1$ case. Namely, we always
assume $w\in\N^{I\times\{-1,0\}}$. Furthermore, we consider only those
$v$ contained in $\N^{I\times\{-\Hf\}}\simeq \N^I$. By putting these restrictions on $w$, $v$
appearing in $\qtChar$
and $\tChar$, we arrive at the \emph{truncated character} maps $\qtChar\trunc$
and $\tChar\trunc$. Proposition
  \ref{prop:mult_pbwTarg} and Theorem \ref{thm:injectiveHom} are still
  valid after truncations, \cf \cite[Proposition 6.1]{HernandezLeclerc09}.

By default, we consider $Q\op$-representations, and the ice quiver $\tQ$ is of level $1$ with $z$-pattern.


For any $m=(m_i)_{i\in I}\in\N^I$, we define $I^m=\oplus_{i\in I} I_i^{\oplus
  m_i}$ (resp. $P^m=\oplus_{i\in I} P_i^{\oplus
  m_i}$) to be the corresponding injective (resp. projective) $Q\op$-representation.



  For any $w,w'\in\N^{I\times\{-1,0\}}$, we define the integer
  $r_{ww'}$ to be given as in \cite[3.2.1]{KimuraQin11}, which are
  determined by Fourier-Deligne-Sato transformations. Then we have 
  \begin{itemize}
  \item $r_{w,w}=1$,
\item for any given $w$, only finitely
    many of the $r_{w,w'}$ are nonzero, 
\item the matrix $(r_{w,w'})$ is
    upper unitriangular with respect to the dominance order.
  \end{itemize}

We define the \emph{almost simple pseudo-module} $\gen(w)$ to be the
element in the Grothendieck group $\quotKGp$ \st we have
\begin{align}
  \gen(w)=\sum_{w'}r_{w,w'}\can(w').
\end{align}

Denote the truncated $qt$-characters $\tChar\trunc(\gen(w))$ by
$\genRedTarg(w)$.

The following is the main theorem of the first part of \cite{KimuraQin11}.

\begin{Thm}[{\cite{KimuraQin11}}]\label{thm:genChar} \label{thm:z_coeff}
  The truncated $qt$-characters of the almost simple pseudo-modules in
  $\redTargSpace$ are given by
  \begin{align}\label{eq:genChar}
    \genRedTarg(w)=\sum_{v}t^{-\dim (\Gr_v \kerMod W)}P_t(\Gr_v \kerMod
    W)Y^{w-C_q v},
  \end{align}
where $\Gr_v \kerMod W$ is the quiver Grassmannian determined by $Q$,
$v$, $w$.
\end{Thm}

\begin{Def}\label{def:pureCoeff}
  For any given $w\in\N^{I\times\{-1,0\}}$, its \emph{pure
    coefficient part} $\pureCoeff w$ is the maximal element in $
  span_{\N}\{e_{i,-1}+e_{i,0}|i\in I\}$  such that $w-\pureCoeff
  w\in\N^{I\times \{-1,0\}}$. We denote the difference $w-\pureCoeff w$
  by $\coeffFree w$ and call it the \emph{coefficient-free part} of
  $w$. Let $\redWSet$ denote the set $\{w|\pureCoeff w=0\}$.
\end{Def}

The truncated characters of almost simple pseudo-modules,
 which are also called the \emph{generic characters}, have the following properties.

\begin{Prop}[{\cite{KimuraQin11}}]\label{prop:clusterMonomial}
1) If $\genRedTarg(w)$ is a quantum cluster monomial, then we have $\gen(w)=\can(w)$.

2) In the notation of Definition \ref{def:pureCoeff}, we have a factorization
  \begin{align}\label{eq:coeffFactorization}
    \genRedTarg(w)=\genRedTarg(\coeffFree w)\cdot\genRedTarg(\pureCoeff w).
  \end{align}
\end{Prop}

\section{From characters to the quantum torus}\label{sec:failure}
In this section, we compare the ring of formal
power series $\targSpace$ in which $qt$-characters live with the
quantum tori $\dTorus$ and $\torus$ in which quantum cluster algebras
live. Although our construction fails to preserve some properties of the products of $qt$-characters, we find that the failures are bearable and controllable. This observation will play a crucial role in later sections.

\subsection{Dual PBW basis elements}

As in Section \ref{sec:frozenQuiver}, let $\tQ$ be an ice quiver
whose principal part $Q$ is acyclic. We construct the cluster category
$\cC_{(\tQ,\tW)}$ (denoted by $\cC_\tQ$ for simplicity), whose shift
functor will be denoted by $[1]$.


Notice that we have the indecomposable 
objects $M_i(t)$, $1\leq i\leq m$, $t\in \sT_n$, in the cluster
category $\cC_\tQ$,
corresponding to the $x$-variables $x_i(t)$, \cf
\cite{Plamondon10b}. We denote the object $M_i(t_0)$ by $T_i$. By
applying a sequence of mutations to $t_0$ corresponding to a sink sequence of the
principal part $Q$, we obtain some vertex $\Sigma t_0\in\sT_n$ together with the new
objects $ M_i(\Sigma t_0)$, $1\leq i\leq n$, which we shall denote by
$I_i$. For any $i\in I$, the object $I_i$ is characterized by the
property that the right $\C \tQ$-module $\Ext^1_{\cC_\tQ}(\oplus_{1\leq j\leq m}T_j,I_i)$ is
supported at the vertices of $Q$ and it can be viewed as the $i$-th
injective right module of $\C Q$ (which is also denoted by $I_i$). 

Following the convention of section \ref{sec:normalization}, for any
$M$ in the cluster category which appears in a triangle 
  \begin{align*}
    T^m\ra T^{m'}\ra M \ra T^m[1]
  \end{align*}
for some for some $m,m'\in\N^m$, its index $\ind M$ is $m'-m\in\Z^m$, \cf \cite{Plamondon10b}.

By abuse of notation, we use
$\ind(\ )$ to denote the map from $\N^{I\times \{-1,0\}}$ to $\Z^n\oplus
\N^{m-n}$ which sends $w=(w_{i,a})\in\N^{I\times \{-1,0\}}$ to
$\ind(\oplus_i I_i^{w_{i,-1}})+\ind(\oplus_i T_i^{w_{i,0}})$.

When the ice quiver $\tQ$ is of level $1$ with $z$-pattern in the sense of \cite{KimuraQin11}, the
index map $\ind(\ )$ is denoted by $\zInd(\ )$ and the compatible pair
by $(\ztB,\zLambda)$. In this section, we shall see how the variation of the
coefficients attached to the principal part $Q$ affects our computation.

\begin{Lem}\label{lem:zPatternInd}
  We have, for $1\leq k\leq n$,
  \begin{align}
    \zInd(e_{k,0})=e_k,\\
\zInd(e_{k,-1})=e_{k+n}-e_k.
  \end{align}
\end{Lem}
\begin{proof}
The first equation is clear. 

In order to calculate the indices, we can choose the potential $\tW$ to
be the sum of the $3$-cycles in the ice quiver $\tQ$ (in the sense of
\cite{DerksenWeymanZelevinsky07}, $\tW$ is a ``nondegenerate'' potential for the ice
  quiver $\tQ$, \cf \cite{BuanIyamaReitenSmith08}).

The index of $I_i$ can be checked directly by applying the mutation
rules to the $g$-vectors along the mutation sequence from $t_0$ to $\Sigma
t_0$. Alternatively, we give a proof using cluster categories as the
follows.

Consider a triangle
\begin{align*}
  T_k\xra{f} T_{k+n}\ra M\ra T_k[1]
\end{align*}
in $\cC_{\tQ}$, where $f$ is non-zero. 

We first show that $M$ satisfies
$\Ext^1_{\cC_\tQ}(M,T_{j+n})=0$, $\forall 1\leq j\leq n$. Applying $\Hom_{\cC_{\tQ}}(\
,T_{j+n})$ to the above triangle, we obtain a long exact sequence
\begin{align*}
  \Hom_{\cC_{\tQ}}(T_{k+n},T_{j+n})\xra{f^*}  \Hom_{\cC_{\tQ}}(T_k,T_{j+n})\ra
  \Hom_{\cC_{\tQ}}(M,T_{j+n}[1])\ra \Hom_{\cC_{\tQ}}(T_{k+n},T_{j+n}[1]).
\end{align*}
The last term is zero since $\oplus_{1\leq i\leq 2n}T_i$ is
rigid. With our choice of the potential, it is easy to check that the map $f^*$ is surjective (in other words, $f$ is a left approximation of $T_k$ in
$\add (T_{n+i},1\leq i\leq n)$ in the sense of
\cite{IyamaYoshino08}). Therefore, we have checked 
$\Ext^1_{\cC_\tQ}(M,T_{j+n})=\Ext^1_{\cC_\tQ}(T_{j+n},M)=0$. 

The
Calabi-Yau reduction (\cf \cite{IyamaYoshino08}) from $\cC_{\tQ}$ to the cluster category $\cC_Q$
of $Q$ sends $M$ to the object with index $-e_k$. It follows that $\Ext^1_{\cC_\tQ}(\oplus_{1\leq i\leq 2n}T_i,M)$ is the
$i$-th injective right module of $\C Q$ and $M$ is isomorphic to the object
$I_i=M_i(\Sigma t_0)$. So we have $\zInd(e_{k,-1})=\ind M=e_{k+n}-e_k$.
\end{proof}


Let $\pr_n$ denote the projection of $\Z^m$ (respectively $\Z^{2n}$)
to the first $n$ coordinates. Then $\pr_n\ind$ is independent of the
choice of the coefficient $\tQ-Q$.

\begin{Lem}\label{lem:zReduction}
We have
\begin{align}
  \zInd(w-C_qv)=\zInd(w)+\ztB v.
\end{align}
\end{Lem}
\begin{proof}
It suffices to check the case of $w=0$ and $v=e_{k,-\Hf}$, $1\leq k\leq
n$. We have
\begin{align*}
  \zInd(-C_qv)&=-\zInd(e_{k,-1}+e_{k,0})+\sum_{i:1\leq i<k}b_{ik}\zInd(e_{i,0})\\& \qquad+\sum_{j: k< j\leq
    n}b_{kj}\zInd(e_{j,-1})\\
&=-e_{k+n}+\sum_{1\leq i<k}b_{ik}e_i+\sum_{k<j\leq n}b_{kj}(e_{j+n}-e_j)\\
&=b_{k+n,k}e_{k+n}+\sum_{1\leq i<k}b_{ik}e_i+\sum_{k<j\leq
  n}b_{j+n,k}e_{j+n}+\sum_{k<j\leq n}b_{jk}e_j\\
&=\ztB e_k.
\end{align*}
\end{proof}

Let $\delta$ be a strictly positive integer. Denote the product of $\delta$ with the rank $m$ identity matrix $I_m$ by $D$. Let
$(\tB,\Lambda)$ be a compatible pair such that $\Lambda(-\tB)=\begin{bmatrix}D\\0 \end{bmatrix}$. 

Define $N=max\{2n,m\}$. Extend $\Lambda$ into an $N\times N$ matrix by putting zero entries
in the extra rows and columns indexed by either $\{m+1,\ldots,N\}$ or
$\{n+1,\ldots,N\}$. We extend the $m\times
n$ matrix $\tB$ and the $2n\times n$ matrix $\ztB$ to $N\times N$
matrices similarly. Notice that
$\Lambda(e_{i},\tB v)=0$, for any $n+1\leq i\leq N$ and
$v\in\N^n$.

\begin{Def}
  The \emph{double quantum torus} is the Laurent polynomial ring
  \begin{align}
    \label{eq:dTorus}
    \dTorus=\qBase[x_1^\pm,\ldots,x_N^\pm,y_1^\pm,\ldots,y_n^\pm],
  \end{align}
together with the twisted product $*$ \st for any $g^1$, $g^2\in
\Z^N$, $v^1$, $v^2\in \Z^n$, we have
\begin{align*}
  x^{g^1}y^{v^1}* x^{g^2}y^{v^2}=q^{\Hf \Lambda(g^1+\tB v^1,g^2+\tB v^2)}x^{g^1+g^2}y^{v^1+v^2}.
\end{align*}
It has the bar involution $\overline{(\ )}$ given by $\overline{(q^{\Hf}x^gy^v)}=q^{-\Hf}x^gy^v$.
\end{Def}

Define the quantum torus $\cT$ to be the $\qBase$-subalgebra of $\dTorus(+,*)$
generated by $x_{1}^\pm$, \ldots, $x_{N}^\pm$. Recall that the coefficient ring is $\qBaseCoeff=\qBase[x_{n+1}^\pm,\ldots,x_{m}^\pm]$. Let $\ZCoeff$ denote its semi-classical limit
under the specialization $q^\Hf\mapsto 1$.

Next, we construct a map from $\targSpace$ to $\dTorus$.

\begin{Def}[Correspondence map]\label{def:correction_map}
The $\Z$-linear map $\cor$ from $\targSpace$ to $\dTorus$ is given by
  \begin{align}
    \label{eq:cor}
    \cor(t^{\lambda}W^wV^v)=q^{\frac{\diag}{2}\lambda}x^{\ind(w)}y^v,
  \end{align}
for any $w$, $v$, and integer $\lambda$.
\end{Def}

Notice that $\cor$ commutes with the bar involutions. Although $\cor$ does \emph{not} commute
with twisted products
$*$, we can measure the failure as below.
\begin{Lem}[Failure \textbf{1}] We have, for any monomials $m^i=W^{w^i}V^{v^i}$, $i=1$, $2$,
  \begin{align}\label{eq:failureCor}
    \begin{split}
      \cor(W^{w^1}
      V^{v^1}*W^{w^2}V^{v^2})&=q^{\frac{\diag}{2}\zLambda(\zInd(w^1),\zInd(w^2))-\Hf\Lambda(\ind(w^1),\ind(w^2))}\\
      &\qquad \cor(W^{w^1}V^{v^1})*\cor(W^{w^2}V^{v^2}).
    \end{split}
  \end{align}
In particular, the $q$-power does not depend on $v^1$, $v^2$.
\end{Lem}

\begin{proof}
Notice that we always have
$\Lambda(\ind w,\tB v)=-\diag \pr_n \ind w\cdot v$, and $\Lambda(\tB v^1,\tB v^2)=\diag\langle v^1,v^2\rangle-\diag\langle v^2,v^1\rangle$, \cf
\cite[5.2.1]{Qin10}. Using Lemma \ref{lem:epsilon} and \ref{lem:zReduction}, we obtain
\begin{align*}
\LHS&=\cor(t^{-\dT(m^1,m^2)+\dT(m^2,m^1)}m^1m^2)\\
&=\cor(t^{-\eMatrix(w^1-C_q v^1,w^2-C_q v^2)}m^1m^2)\\
&=q^{\frac{\diag}{2}\zLambda(\zInd(w^1-C_q v^1),\zInd(w^2-C_q
  v^2))}X^{\ind w^1}Y^{v^1}X^{\ind w^2}Y^{v^2}\\
&=q^{\frac{\diag}{2}\zLambda(\zInd w^1+\ztB v_1,\zInd w^2+\ztB
  v^2)}q^{-\Hf\Lambda(\ind w^1+\tB v^1,\ind w^2+\tB v^2)}\\
&\qquad X^{\ind
  w^1}Y^{v^1}*X^{\ind w^2}Y^{v^2}\\
&=q^{\frac{\diag}{2}\zLambda(\zInd w^1,\zInd w^2)}q^{-\Hf\Lambda(\ind
  w^1,\ind w^2)}X^{\ind w^1}Y^{v^1}*X^{\ind w^2}Y^{v^2}\\
&=\RHS.
\end{align*}

\end{proof}

\begin{Def}
We define the $\qBase$-linear map $\contr$ from $\dTorus$ to $\torus$ \st it sends $x^gy^v$ to $x^{g+\tB v}$.
\end{Def}
Notice that $\contr$ is an algebra homomorphism with respect to both
the usual products and the twisted products.

Let us define
\begin{align}
  \pbwDTorus(v,w)=\cor\qtChar\trunc(\pbw(v,w)),\\
\pbwTorus(v,w)=\contr\pbwDTorus(v,w).
\end{align}
When $v=0$, we also
denote them by $\pbwDTorus(w)$ and $\pbwTorus(w)$. Explicitly, we have
  \begin{align*}
    \pbwDTorus(w)&=\sum_vP_{q^{\frac{\diag}{2}}}(\lag(v,w))q^{-\frac{\diag}{2}
      \dim\grProjQuot(v,w)}x^{\ind(w)}y^v\\
&=\sum_v\cor(\langle M(0,w),\pi(v,w)\rangle) x^{\ind(w)}y^v.
  \end{align*}

Recall that we have a similar map from $\targSpace$ to
$\redTargSpace$, which is also denoted by $\contr$. In general, we
do \emph{not} have $\contr\qtChar=\tChar\contr$. This failure is measured by the following result.
\begin{Lem}[Failure \textbf{2}]
We have
  \begin{align}\label{eq:failureContr}
    \pbwTorus(v,w)=\pbwTorus(w-C_qv)x^{\tB v+\ind C_q v}.
  \end{align}

\end{Lem}
\begin{proof}Straightforward calculation shows
  \begin{align*}
    \pbwTorus(v,w)&=\sum_{v':v'-v\in\N^I}\contr\cor(\langle M(v,w),\pi(v',w)\rangle)
    x^{\ind(w)}x^{\tB v'}.\\
    &=\sum_{v'-v}\contr\cor(\langle M(0,w-C_qv),\pi(v'-v,w)\rangle)
    x^{\ind(w-C_q v)}x^{\tB (v'-v)}\\
    &\qquad x^{\tB v+\ind(w)-\ind(w-C_qv)}.\\
    &=\pbwTorus(w-C_qv)x^{\tB v+\ind C_q v}.
  \end{align*}
\end{proof}
Notice that the \emph{correction factor} $x^{\tB v+\ind C_q v}$ is contained
in $\qBaseCoeff$.

\begin{Prop}
  \label{prop:bar_inv}
We have
\begin{align}
  \label{eq:bar_pbwTorus}
  \overline{\pbwTorus(w)}=\sum_{(v,w)\leq
    (0,w)}u_{(0,w),(v,w)}x^{\tB v+\ind C_qv}\pbwTorus(w-C_qv).
\end{align}
\end{Prop}
\begin{proof}
If we apply $\contr\cor$ to \eqref{eq:bar_pbwTarg}, the statement follows
from \eqref{eq:failureContr}.
\end{proof}

\begin{Prop}
  \label{prop:mult_pbw}
Fix $w^1$ and $w^2$. If for all $i,j\in I$ and
$a>b\in\Z$, either $(w^1)_i(a)$ or $(w^2)_j(b)$ vanishes, then the multiplicative property holds:
\begin{align*}
\pbwTorus(w^2)*\pbwTorus(w^1)&=q^{\Hf
  \Lambda(\ind(w^2),\ind(w^1))-\frac{\diag}{2}\zLambda(\zInd(w^2),\zInd(w^1))}\\
&\qquad q^{\frac{\diag}{2}\eMatrix(w^1,w^2)} \pbwTorus(w^1+w^2).
\end{align*}
\end{Prop}

\begin{proof}
Using \eqref{eq:failureCor}, \eqref{eq:failureContr}, and Proposition \ref{prop:mult_pbwTarg}, we obtain
  \begin{align*}
    \pbwTorus(w^2)*\pbwTorus(w^1)&=\contr(\cor\qtChar\trunc(\pbw(w^2))*\cor\qtChar\trunc(\pbw(w^1)))\\
    &=q^{\Hf\Lambda(\ind(w^2),\ind(w^1))-\frac{\diag}{2}\zLambda(\zInd(w^2),\zInd(w^1))}\\&\qquad
    \contr\cor(\qtChar\trunc(\pbw(w^2
    ))*\qtChar\trunc(\pbw(w^1)))\\
&=\RHS.
  \end{align*}
\end{proof}

Let $\subQClAlg$ denote the vector space spanned by the standard
basis elements $\pbwTorus(w)$ over $\qBaseCoeff$. The following is the main result of this section.

\begin{Prop}\label{prop:mult_closed}
The vector space $\subQClAlg$ is closed under
the involution $\overline{(\ )}$ and the twisted
  multiplication $*$.
\end{Prop}
\begin{proof}
The first assertion follows from Proposition \ref{prop:bar_inv}. It
remains to verify the second one.

Recall that $w_i(a)$ is the $(i,a)$-th component of $w$, and
$w(a)=\oplus_i w_i(a)$, where $i\in
I$, $a\in\{-1,0\}$. We have $\pbw(w_i(a))=\can(w_i(a))$. Consequently,
$\pbwTorus(w_i(a))=\contr\cor\qtChar\trunc(\can(w_i(a)))$ is
bar-invariant. Applying Proposition \ref{prop:mult_pbw}, we obtain that
$\pbwTorus(w(a))$ is bar-invariant.

Use Propositions \ref{prop:bar_inv} and \ref{prop:mult_pbw}. For any two elements $\pbwTorus(w^1)$, $\pbwTorus(w^2)$, up to
specified invertible elements in $\qBase$,
$\pbwTorus(w^1)*\pbwTorus(w^2)$ becomes
\begin{align*}
&\pbwTorus(w^1(0))*\pbwTorus(w^1(-1))*\pbwTorus(w^2(0))*\pbwTorus(w^2(-1))\\
=&\pbwTorus(w^1(0))*\overline{\overline{\pbwTorus(w^2(0))}*\overline{\pbwTorus(w^1(-1))}}*\pbwTorus(w^2(-1))\\
=&\pbwTorus(w^1(0))*\overline{\pbwTorus(w^2(0))*\pbwTorus(w^1(-1))}*\pbwTorus(w^2(-1))\\
=&\pbwTorus(w^1(0))*\overline{\pbwTorus(w^2(0)+w^1(-1))}*\pbwTorus(w^2(-1))\\
=&\pbwTorus(w^1(0))*(\sum_{w'}u_{w,w'}x^{\tB
  v+\ind C_q v}\pbwTorus(w'))*\pbwTorus(w^2(-1)),
\end{align*}
where we write $w=w^2(0)+w^1(-1)$, $w'=w-C_qv$,
$u_{w,w'}=u_{(0,w),(v,w)}$. The monomial $x^{\tB v+\ind C_q v}$ is contained in
$\qBaseCoeff$ and quasi-commutes with the other factors. Therefore, up to
specified invertible elements in $\qBase$, the above result becomes
\begin{align*}
&\sum_{w'}u_{w,w'}x^{\tB v+\ind C_q
  v}\pbwTorus(w^1(0))*\pbwTorus(w'(0))*\pbwTorus(w'(-1))\\
&\qquad *\pbwTorus(w^2(-1))\\
=&\sum_{w'}u_{w,w'}x^{\tB v+\ind C_q v}\pbwTorus(w^1(0)+w'(0)+w'(-1)+w^2(-1)).
\end{align*}
\end{proof}
\begin{Rem}
  \begin{enumerate}
\item The above proof shows that the twisted
  product $*$ of the standard basis elements is determined by
  Proposition \ref{prop:mult_pbwTarg} and \ref{prop:bar_inv}.
  \item Notice that the map $\contr\cor\qtChar\trunc$ is algebraic when $\tQ$
    is of $z$-pattern. So Theorem \ref{thm:correction} below provides
    another proof to Proposition \ref{prop:mult_closed}.
  \end{enumerate}
\end{Rem}

\subsection{Dual canonical basis elements}
For any $l$-dominant pair $(v,w)$, define $\canTorus(v,w)=\contr\cor\qtChar\trunc(\can(v,w))$ and
denote it by $\canTorus(w)$ when $v=0$. Then $\canTorus(w)$ is given by
  \begin{align*}
    \canTorus(w)=\sum_{v}a_{v,0;w}(q^{\frac{\diag}{2}}) x^{\ind(w)+\tB v},
  \end{align*}
where the Laurent polynomials $a_{v,0;w}(t)=\sum_{d\in\Z}a_{v,0;w}^dt^d$ are given by \ref{eq:decomposition}.
We shall see all the quantum cluster monomials essentially take this form.

\begin{Lem} We have
  \begin{align*}
    \canTorus(v,w)=\canTorus(w-C_qv)x^{\tB v+\ind C_q v}.
  \end{align*}
\end{Lem}
\begin{proof}
1) Similar to \eqref{eq:failureContr}, we compute
  \begin{align*}
    \canTorus(v,w)&=\sum_{v'}\contr\cor(\langle L(v,w),\pi(v',w) \rangle) x^{\ind(w)}x^{\tB v'}\\
&=\sum_{v'-v}\contr\cor(\langle L(0,w-C_q v), \pi(v'-v,w- C_q v)\rangle) x^{\ind(w-C_q v)}x^{\tB (v'-v)}x^{\tB v+\ind
  (w)-\ind(w-C_q v)}\\
&=\canTorus(w-C_qv)x^{\tB v+\ind C_q v}.
  \end{align*}
\end{proof}

We already know that
$\pbw(0,w)=\sum_{v}Z_{(0,w)(v,w)}\can(v,w)$. If we apply
$\contr\cor\qtChar\trunc$ to it, we obtain
\begin{align*}
\pbwTorus(w)=\sum_{v}Z_{(0,w)(v,w)}\canTorus(w-C_qv)x^{\tB v+\ind C_q v}.
\end{align*}

\subsection{Generic basis elements}
\label{sec:genericBasis}

The generic basis elements $\genTorus(w)$ are defined to be
\begin{align}\label{eq:genTorusElem}
  \genTorus(w)=\sum_{v:(v,w)\text{ is $l$-dominant }}r_{w,w'}x^{\tB v+\ind C_q
    v}\canTorus(w-C_qv),
\end{align}
where $w'=w-C_qv$, and the integer $r_{w,w'}$ is given in \cite[3.2.1]{KimuraQin11}. We have
\begin{align*}
  \genTorus(w)&=\sum_{v}r_{w,w-C_qv}\canTorus(w-C_qv)x^{\tB
    v+\ind C_q v}\\
&=\sum_{v,v'}r_{w,w-C_q v}a_{v',0;w-C_q v}(q^{\frac{\diag}{2}})x^{\ind (w-C_q v)+\tB v'}x^{\tB
    v+\ind C_q v}\\
&=\sum_{v,v'} r_{w,w-C_q v}a_{v'+v,v;w}(q^{\frac{\diag}{2}})x^{\ind (w)}x^{\tB (v+v')}\\
&=\sum_{v+v'=v''}b_{v''}(q^{\frac{\diag}{2}})x^{\ind (w)}x^{\tB v''},
\end{align*}
where the coefficient $b_{v''}=\sum_{(v,w)\text{ is
    $l$-dominant}} r_{w,w-C_q v}\cdot a_{v'',v;w}$.

In particular, $b_{v''}$ is independent of the
frozen pattern. 

\section{Bases of acyclic quantum cluster algebras}\label{sec:bases}

We keep the assumptions on the ice quiver $\tQ$ and the vectors $v$, $w$ as
in Section \ref{sec:failure}.

\subsection{Generic basis}
\label{sec:genBasis}
The results in Section \ref{sec:pseudoModules} and \ref{sec:genBasis} imply the following theorem.

\begin{Thm}\label{thm:anyPattern}
1) The generic basis elements have the following expansion
\begin{align}
  \genTorus(w)=\sum_{v}P_{q^{\frac{\diag}{2}}}(\Gr_v (\kerMod W))q^{-\frac{\diag}{2}\dim
    \Gr_v(\kerMod W)}x^{\ind w+\tB v}.
\end{align}

2) We have a factorization
  \begin{align}\label{eq:factorization}
    \genTorus(w)=\genTorus(\coeffFree w)\cdot\genTorus(\pureCoeff w).  
  \end{align}

3) If $\genTorus(w)$ becomes a quantum cluster monomial when we choose
$\tQ$ to be of $z$-pattern, then we have $\genTorus(w)=\canTorus(w)$. 
\end{Thm}

For any given $w\in\redWSet$, we choose a generic map $f:I^{w(-1)}\ra
I^{w(0)}$ and define the object $O(w)$ to be $\mathrm{Cone}(f)[-1]$ in the
cluster category. So we
  have the following triangle
  \begin{align*}
    O(w)\ra I^{w(-1)}\xra{f} I^{w(0)}\ra O(w)[1].
  \end{align*}
We define the associated \emph{generic quantum
  cluster characters} to be
\begin{align}
  \begin{split}
    \genCl(w)=&\genTorus(w)\cdot x^{\ind O(w)-\ind w}\\
=&\sum_{v}P_{q^{\Hf\diag}}(\Gr_v (\kerMod W))q^{-\Hf\dim \Gr_v(\kerMod W)}x^{\ind O(w)+\tB v}.
  \end{split}
\end{align}

\begin{Rem}
We can denote $\genCl(w)$ by $X_{O(w)}$. This definition naturally
  generalizes the quantum cluster character formula in
  \cite[Definition 1.2.1]{Qin10} to generic objects.

It is not clear if one can extend this
quantum character to arbitrary objects of the presentable cluster
category in a reasonable way, for example, \st its image is still contained
in $\qClAlg$. Some results for Dynkin quivers are discussed in \cite{Ding10}.
\end{Rem}

\begin{Prop}\label{prop:compareIndices}
For any $w\in\redWSet$, we have $\pr_n\ind O(w)=\pr_n\ind w$. When the
ice quiver $\tQ$ is the level $1$ ice quiver with $z$-pattern, we have $\ind O(w)=\ind w$.
\end{Prop}
\begin{proof}

Because $\pr_n\ind O(w)$ and $\pr_n\ind w$ do not depend on the
coefficients $\tQ-Q$, it suffices to prove the second statement.

    First, the index $\ind w$ is linear in $w$.

    Second, the index $\ind O(w)$ is linear with respects to the
    components in the canonical decomposition of the generic objects
    $O(w)$, which are either generic modules or the object $P_i[1]$,
    $i\in I$.

    Therefore, it suffices to study the indices of modules. Let $O(w)=M$ be
    any $Q\op$-module with the minimal injective resolution
    \begin{align}\label{eq:minInjResolution}
      0\ra M\xra{f} I^{w(-1)}\xra{g} I^{w(0)}\ra 0.
    \end{align}
Let $B$ denote the Jacobi algebra of $(\tQ,\tW)$. View the resolution as a short exact sequence in the category of
$B\op$-modules. Denote the simple $B\op$-modules by $S_j$, $1\leq
j\leq 2n$. If we apply $\Hom_{B\op-\mod}({S_j},\ )$ to the
above short exact sequence, we obtain a long exact sequence
\begin{align*}
\ldots  &\ra \Hom_{B\op-\mod}(S_j,M)\xra{f_0} \Hom_{B\op-\mod}(S_j, I^{w(-1)})\ra
\Hom_{B\op-\mod}(S_j, I^{w(0)})\\
&\xra{w_0} \Ext^1_{B\op-\mod}(S_j, M)\xra{g_1} \Ext^1_{B\op-\mod}(S_j, I^{w(-1)})\ra \Ext^1_{B\op-\mod}(S_j,I^{w(0)} )\ra\ldots.
\end{align*}
Then $(\ind M)_j=-\dim\Hom_{B\op-\mod}(S_j,M)+
\dim\Ext^1_{B\op-\mod}(S_j, M)$, \cf \cite{Palu08a}.

Because $M$ is supported on the principal
part $Q$, $\dim\Hom_{B\op-\mod}(S_j,M)$ vanishes unless $1\leq j\leq
n$. The fact that \eqref{eq:minInjResolution} is the minimal injective
resolution of $Q\op$-modules implies $f_0$ is an isomorphism. It
follows that $w_0$ is injective. 

When the quiver is of $z$-pattern, we have $\ind I_i=-e_i+e_{i+n}$,
\cf Lemma \ref{lem:zPatternInd}. Therefore
$\Ext^1_{B\op-\mod}(S_j,I^{w(-1)} )$ equals $w_{j-n}(-1)$ if $j>n$ and
vanishes elsewhere. Furthermore, the fact that \eqref{eq:minInjResolution} is the minimal injective
resolution of the $Q\op$-module $M$ implies that for any $i\in I$,
either $\dim\Ext^1_{B\op-\mod}(S_{i+n},I^{w(-1)} )=w_i(-1)$ or
$\dim\Ext^1_{B\op-\mod}(S_{i+n},I^{w(0)} )=w_i(0)$ is zero. Therefore, $g_1$ is surjective.

Therefore, we have
\begin{align*}
  (\ind M)_j&=-\dim\Hom_{B\op-\mod}(S_j, I^{w(-1)})+\dim
  \Hom_{B\op-\mod}(S_j,I^{w(0)})\\
&\qquad +\dim\Ext^1_{B\op-\mod}(S_j,I^{w(0)} )\\
&=\left\{
\begin{array}{ll}
-w_j(-1)+w_j(0) & \textrm{if $j\leq n$}\\
w_j(0)+w_{j-n}(-1) & \textrm{if $j>n$},
\end{array} \right.\\
&=(\ind w)_j.
\end{align*}
\end{proof}

\begin{Rem}[Failure \textbf{3}]
  For general $\tQ$, we have 
  \begin{align}\label{eq:failureInd}
    \ind w\neq\ind O(w).
  \end{align}

\end{Rem}

\begin{Lem}\label{lem:coeff_commute}
The elements in the coefficient ring $\qBaseCoeff$ quasi-commute with
$\pbwTorus(w)$, $\canTorus(w)$, $\genTorus(w)$, $w\in\N^{I\times \{-1,0\}}$.
\end{Lem}
\begin{proof}
It follows from the fact that $x_i$, $n<i\leq N$, commutes with $x^{\tB
v}$ for any $v\in\N^n$.
\end{proof}

\begin{Prop}\label{prop:differentIndex}
The set $\redGenBasis$ is a $\qBaseCoeff$-basis of the
  vector space $\subQClAlg$.
\end{Prop}
\begin{proof}
Notice that the matrix $(r_{w,w'})$ is upper unitriangular. Furthermore, each element $\genTorus(w)$ divides into the composition of the
  coefficient free part $\genTorus(\coeffFree w)$ and the pure coefficient
  part $\genTorus(\pureCoeff w)\in\qBaseCoeff$. Therefore, $\subQClAlg$
  is spanned by $\redGenBasis$ over $\qBaseCoeff$.

Take any vectors $w$, $w'$ in $\redWSet$, \st $w\neq w'$. By proposition
\ref{prop:compareIndices}, the vector $\pr_n \zInd w$ is the
index of a generic object in the cluster
category of $Q$. Furthermore, we
have $\pr_n \ind w\neq \pr_n \ind w'$, \cf the proof of Proposition
\ref{prop:compareIndices} or \cite{Plamondon10a}. Because $\ind w$ is the leading term of
$\genTorus(w)$ and $\tB$ is of full rank, the set $\redGenBasis$ is $\qBaseCoeff$-linearly independent.
\end{proof}

\begin{Prop}\label{prop:subRing}
  $\subQClAlg$ equals $\qClAlg$.
\end{Prop}
\begin{proof}
$\subQClAlg$ is a subspace of $\qClAlg$ because
the latter contains $\pbwBasis$. If we take any two elements $\genTorus(\coeffFree w^1)* f_1$,
$\genTorus(\coeffFree w^2)* f_2$, where $f_1,f_2\in\qBaseCoeff$, it
follows from Proposition
\ref{prop:mult_closed} that their twisted product still belongs to
$\subQClAlg$. Therefore, $\subQClAlg$ is a subalgebra of $\qClAlg$
with respect to the twisted product. But it contains all the quantum cluster variables. Therefore, the two algebras must agree.
\end{proof}

\begin{Thm}\label{thm:genBasis}
    $\clGenBasis$ is a $\qBaseCoeff$-basis of $\qClAlg$. It is called
    the \emph{generic basis}. Furthermore,
    it contains all the quantum cluster monomials.
\end{Thm}
\begin{proof}
The statement follows from Proposition \ref{prop:differentIndex}
and \ref{prop:subRing}. 
\end{proof}

\subsection{Dual canonical basis and dual PBW basis}
\label{sec:otherBases}

Similar to the treatment of generic quantum cluster characters, for each
$w\in\redWSet$, we normalize the quantum torus elements $\pbwTorus(w)$ and
$\canTorus(w)$ by defining
\begin{align}
      \pbwCl(w)=&\pbwTorus(w)\cdot x^{\ind O(w)-\ind w},\\
    \canCl(w)=&\canTorus(w)\cdot x^{\ind O(w)-\ind w}.\\
\end{align}

\begin{Thm}\label{thm:otherBases}
The sets $\clPbwBasis$, $\clCanBasis$ are $\qBaseCoeff$-bases of
$\qClAlg$. They are called the \emph{dual PBW basis }and \emph{the dual canonical
basis} of the quantum cluster algebra $\qClAlg$
respectively. Furthermore, all the quantum cluster monomials are
contained in $\clCanBasis$.
\end{Thm}
\begin{proof}
It suffices to show that $\redPbwBasis$, $\redCanBasis$ are $\qBaseCoeff$-bases of
$\qClAlg$. Applying the truncated character map to
\eqref{eq:pbwToCan}, we obtain
\begin{align}
  \pbwTorus(w)=\canTorus(w)+\sum_{(v,w)\text{ is
      $l$-dominant}}Z_{w,w'}(q^{\frac{\diag}{2}})x^{\tB v+\ind C_q v}\canTorus(w'),
\end{align}
where $w'=w-C_q v$ and $Z_{w,w'}(t)=Z_{0,v;w}(t)\in
t^{-1}\Z[t^{-1}]$. Denote the matrix of the coefficients
$Z_{w,w'}(q^{\frac{\diag}{2}})x^{\tB v+\ind C_q v}$ by $\tilde{Z}(q^{\frac{\diag}{2}})$.

Similarly, denote the matrix of the coefficients $r_{w,w'}x^{\tB
  v+\ind C_q v}$ in \eqref{eq:genTorusElem} by $R$ and its inverse
by $R^{-1}$.

Define the $\redWSet\times \redWSet$
matrix $\coeffFree R^{-1}$ \st its entry in position $(w,w')$ is
\begin{align*}
(\coeffFree R^{-1})_{w,w'}=\sum_{w''\in \N^{I\times
    \{-1,0\}}:\coeffFree w''=w'}R^{-1}_{w,w''}x^{\ind(\pureCoeff w'')}.
\end{align*}
Denote the product $\tilde{Z}(q^{\frac{\diag}{2}})R^{-1}$ by
$S(q^{\frac{\diag}{2}})$. Similarly, the $\redWSet\times \redWSet$
matrix $\coeffFree S(q^{\frac{\diag}{2}})$ is defined
\st its entry in position $(w,w')$ is
\begin{align*}
(\coeffFree S(q^{\frac{\diag}{2}}))_{w,w'}=\sum_{w''\in \N^{I\times
    \{-1,0\}}:\coeffFree w''=w'}S(q^{\frac{\diag}{2}})_{w,w''}x^{\ind(\pureCoeff w'')}.
\end{align*}

   The matrix transition between $\redGenBasis$ and
$\redCanBasis$ is given by the matrix $\coeffFree R^{-1}$. The matrix
transition between $\redGenBasis$ and $\redPbwBasis$ is given by
the matrix $\coeffFree S(q^{\frac{\diag}{2}})$.

With respect to the dominance order, the matrices $R^{-1}$ and $S$ are
upper unitriangular. It follows that $\coeffFree R^{-1}$ and
$\coeffFree S$ are upper triangular. Notice that, if for $w$, $w'$ in $\N^{I\times \{-1,0\}}$, we have $w'\leq w$ with
respect to the dominance order, then $w'$ must be equal to $w$. It follows that $\coeffFree R^{-1}$ and
$\coeffFree S$ are upper unitriangular. Therefore, we obtain the statements from Theorem \ref{thm:genBasis}.
\end{proof}

\begin{Rem}\label{rem:transitionChange} The transition matrices vary little when the choice of the
  coefficients and quantization change:

1) The matrix $Z$ can be solved recursively by a
combinatorial algorithm,
  \cf \cite[7.10]{Lusztig90} or \cite[Section 8]{Nakajima04}. $Z$ only
  depends on $Q$. Also, the
  geometrically defined integer matrix $(r_{w,w'})$ only depends on $Q$. The set
  $\{w':\coeffFree w'=w''\}$ for given $w''$ can be calculated easily.

2) We have $\tilde{Z}(q^{\frac{\diag}{2}})=Z(q^{\frac{\diag}{2}})$ and $R=(r_{w,w'})$ if the quiver is of
$z$-pattern.
\end{Rem}

\subsection{Structure constants}
\label{sec:structureConstant}

In \cite{Kimura10}, Kimura studied the factorization of the dual canonical basis $\mathbf{B}^{\mathrm{up}}$
of a certain quantum unipotent subgroup. Because in \cite{KimuraQin11}, the authors identified $\can(w)$ with
elements in $\mathbf{B}^{\mathrm{up}}$, we can translate his result into
our setting.

\begin{Thm}[{\cite[Theorem 6.21]{Kimura10}}]\label{thm:canFactorization}
Up to $t$-powers, we have the factorization of simples
\begin{align}\label{eq:canFactorization}
  \can(w)=\can(\coeffFree w)\cdot \can(\pureCoeff w).
\end{align}
\end{Thm}

For any $w^1$, $w^2$, $w^3$ in $\redWSet$, define an element in $\qBaseCoeff$
  \begin{align}
    \redCanStr_{w^1,w^2;\tB}^{w^3}=x^{\ind O(w^1)+\ind O(w^2)-\ind
      O(w^3)-\ind (w^1+w^2-w^3)}\sum_{w:\coeffFree
      w=w^3}\canStr_{w^1,w^2}^{w}x^{\ind
      \pureCoeff w}x^{\tB v+\ind C_q v}
  \end{align}
where $v$ is determined by 
\begin{align}\label{eq:u}
  w=w^1+w^2-C_q v
\end{align}
and the integers $b_{w^1,w^2}^w$ by \eqref{eq:otimes}.

\begin{Thm}[Positive basis]\label{thm:positiveBase}
  The structure constants of the dual canonical basis $\redCanBasis$ are
  given by, for any $w^1$, $w^2$ in $\redWSet$,
\begin{align*}
    \canCl(w^1)*\canCl(w^2)&=q^{\Hf\Lambda(\ind(w^1),\ind(w^2))-\frac{\diag}{2}\zLambda(\zInd(w^1),\zInd(w^2))}\\
&\qquad
\cdot\sum_{w^3\in\redWSet}\redCanStr_{w^1,w^2;\tB}^{w^3}(q^{\frac{\diag}{2}})\canCl(w^3).
\end{align*} In particular, they are contained in $\N[q^{\pm\frac{1}{2}}][x_{n+1}^\pm,\ldots,x_m^\pm]$.
\end{Thm}
\begin{proof}
  First consider the quantum cluster algebra $\qClAlg$ associated with
  $(\ztB,\zLambda)$. Then we have
  $\contr\qtChar=\tChar\contr$. The truncated character $\tChar\trunc$ is algebraic from $R_t$
  to $\dTorus$. Therefore, the $\tBase$-linear independent set
  $\canBasis$ in $\torus$ is closed under the twisted product $*$ and
  the corresponding positive structure constants are $b_{w^1,w^2}^{w}$
  as given by
  \eqref{eq:otimes}. Using Theorem \ref{thm:canFactorization}, we
  obtain the structure constants of the dual canonical basis
$\redCanBasis$.

Finally, the correction technique in Theorem
\ref{thm:correction} implies the above result for any acyclic quantum
cluster algebra.
\end{proof}

  We obtain that, whatever the coefficient pattern $\tQ-Q$ and the
  quantization $\Lambda$ we choose, the quantum cluster algebras containing acyclic
  seeds are \emph{strongly positive}: they admit a basis in which
  the structure constants are positive.

The following important consequence of the deformed monoidal
categorification is an easy generalization of the first main result of the author's joint work with Kimura
\cite{KimuraQin11} for arbitrary choice of coefficients and quantization.

\begin{Cor}(Quantum positivity {\cite{KimuraQin11}})\label{cor:positivity}
Any quantum cluster monomial $m$ can be written as a Laurent
polynomial of the quantum cluster variables $x_i$, $1\leq i\leq n$, in
any given seed with coefficients in $\N[q^{\pm\frac{\diag}{2}},x_{n+1}^\pm,\ldots, x_{m}^\pm]$.
\end{Cor}
\begin{proof}
By the quantum Laurent phenomenon, we have
\begin{align*}
  m=\frac{\sum_{m_*=(m_i)} c_{m_*}\prod_{1\leq i\leq n}
    x_i^{m_i}}{\prod_i x_i^{d_i}},
\end{align*}
where $m_*=(m_i)_{i\in I}$, $d_*=(d_i)_{i\in I}$ are sequences of nonnegative integers and the
coefficients $c_{m_*}$ are contained in $\qBaseCoeff$. Notice that we
use the usual product $\cdot$ in this expression.

The quantum cluster monomial $m$ equals $\canCl(w)$ for some
$w$. Also, the
quantum $X$-variable $x_i$, $1\leq i\leq m$, equals $\canCl(w_i)$ for
some $w_i$. We can rewrite the above equation as
\begin{align*}
\sum_{m_*=(m_i)} c_{m_*}\canCl(\sum_i m_iw_i)&=\prod_i\canCl(w_i)^{d_i} \cdot \canCl(w)\\
&=q^{-\Hf\Lambda(\ind(\sum_i d_iw^i),\ind(w))}\canCl(\sum_i d_iw_i) * \canCl(w).
\end{align*}

The statement follows from Theorem \eqref{thm:positiveBase}.
\end{proof}

The almost
simple pseudo-modules $\gen(w)$ introduced in Section
\ref{sec:pseudoModules} form a $\qBase$-basis of $\quotKGp$. Notice that, when the ice quiver is
of $z$-pattern, we have
$\contr\cor\tChar\trunc\gen(w)=\genTorus(w)=\genCl(w)$. Denote the structure constants of this basis
by $\genStr_{w^1,w^2}^{w}(t)$: for any $w^1$, $w^2$ in
$\N^{I\times\{-1,0\}}$, we have
\begin{align}\label{eq:genStr}
  \gen(w^1)\otimes\gen(w^2)=\sum_{w}\genStr_{w^1,w^2}^w\gen(w).
\end{align}
Because the matrix $(r_{w,w'})$ and the matrix of the structure constants of $\canBasis$
are upper unitriangular, the matrix $(\genStr_{w^1,w^2}^w)$ is upper
unitriangular as well. Therefore, for every nonzero term in the sum, we always have $w=w^1+w^2-C_q v$ for
some $v$.

For any $w^1$, $w^2$, $w^3$ in $\redWSet$, define an element in $\qBaseCoeff$
  \begin{align}
    \redGenStr_{w^1,w^2;\tB}^{w^3}=x^{\ind O(w^1)+\ind O(w^2)-\ind
      O(w^3)-\ind (w^1+w^2-w^3)}\sum_{w:\coeffFree
      w=w^3}\genStr_{w^1,w^2}^{w}x^{\ind
      \pureCoeff w}x^{\tB v+\ind C_q v}
  \end{align}
where $v$ is determined by $w=w^1+w^2-C_q v$. 

\begin{Thm}
  The structure constants of the generic basis $\redGenBasis$ are
  given by, for any $w^1$, $w^2$ in $\redWSet$,
\begin{align*}
    \genCl(w^1)*\genCl(w^2)&=q^{\Hf\Lambda(\ind(w^1),\ind(w^2))-\frac{\diag}{2}\zLambda(\zInd(w^1),\zInd(w^2))}\\
&\qquad \cdot\sum_{w^3\in\redWSet}\redGenStr_{w^1,w^2;\tB}^{w^3}(q^{\frac{\diag}{2}})\genCl(w^3).
\end{align*}
\end{Thm}
\begin{proof}
This theorem is a consequence of
\eqref{eq:factorization} and Theorem \ref{thm:correction}. The
proof is similar to that of Theorem \ref{thm:positiveBase}.
\end{proof}

\section{Correction technique}\label{sec:correction}

\subsection{Corrections of algebraic relations}

    Given any (quantum) cluster algebra, if some basis of it is known,
    we want to ask what happens if we change the coefficients and the
    quantization. As we have seen in Theorem \ref{thm:genBasis}
    Theorem \ref{thm:otherBases} and Remark
    \ref{rem:transitionChange}, the bases and their transition
    matrices vary little when the coefficient type $\tQ-Q$ and the
    quantization $\Lambda$ change. We shall
    show that this phenomenon is true in general, by generalizing the correction factors in the previous
    failures \eqref{eq:failureCor} \eqref{eq:failureContr} \eqref{eq:failureInd}.

Let $n$, $m\one$, $m\two$
    be three integers, \st $0<n\leq m\one,m\two$. For $i=1$, $2$, let
    $\tB\typeI$ be an $m\typeI\times n$ matrix. Let $\Lambda\typeI$ be
    an $m\typeI\times m\typeI$ skew-symmetric integer matrix.
    \begin{Def}[weakly compatible pair]
      The pair $(\tB\typeI,\Lambda\typeI)$ is called \emph{weakly
  compatible} if there
    exists an $n\times n$ integer matrix $D\typeI$ \st we have
    \begin{align}\label{eq:weaklyCompatible}
      \Lambda\typeI(-\tB\typeI)=\begin{bmatrix}D\typeI\\0 \end{bmatrix}.
    \end{align}
    \end{Def}

 Assume the matrices $\tB\typeI$ have the common principal part $B$,
 the pairs $(\tB\typeI,\Lambda\typeI)$ are weakly compatible, and
 $D\two=\diag D\one$ for some integer $\diag$. A priori, $\rank D\one$ is
 no less than $\rank D\two$.

Define the associated quantum tori
    \begin{align*}
      \torus\typeI=\torus\typeI(\Lambda\typeI)=\qBase[x_1^\pm,\ldots,x_{m\typeI}^\pm]
    \end{align*}
    \st twisted product is determined by $\Lambda\typeI$.

    Let $max(m\one,m\two)$ be denoted by $m$. As in \eqref{eq:dTorus},
    enlarge $\tB\one$, $\tB\two$ into $m\times n$ matrices and $\Lambda\two$
    into an $m\times m$ matrix by adding zero entries, and define the
    associated  enlarged quantum torus
    \begin{align*}
      \torus=\torus(\Lambda\two)=\qBase[x_1^\pm,\ldots,x_{m}^\pm],
    \end{align*}
    \st its twisted product is determined by the enlarged matrix
    $\Lambda\two$.

    For each $i$, and for any integer $s\geq 3$, integer $1\leq j\leq
    s$, vector $g\typeI_j\in\Z^{n_i}$, and polynomial
    \begin{align*}
      F_j(t;y_1,\ldots,y_n)=\sum_{v\in\N^n}b_j(t;v)y^v\in\tBase[y_1,\ldots,y_n],
    \end{align*}
    where $t$, $y_1$, \ldots, $y_n$ are indeterminates,
    $b_j(t;v)\in\tBase$, we define the following element in the quantum
    torus $\torus\one$:
    \begin{align*}
      M\one_j&=x^{g\one_j}F_j(q^{\Hf};x^{\tB\one
        e_1},\ldots,x^{\tB\one e_n})\\
      &=x^{g\one_j}\sum_{v\in\N^n}b_j(q^{\Hf};v)x^{\tB\one
        v},
    \end{align*}
    where $e_k$ is the $k$-th unite vector in $\Z^n$,  and a similar element in $\torus\two$:
    \begin{align*}
      M\two_j&=x^{g\two_j}F_j(q^{\frac{\diag}{2}};x^{\tB\two
        e_1},\ldots,x^{\tB\two e_n})\\
&=x^{g\two_j}\sum_{v\in\N^n}b_j(q^{\frac{\diag}{2}};v)x^{\tB\two
        v}.
    \end{align*}
Assume the first $n$ coordinates of $g\one_j$ and $g\two_j$ are equal,
\ie we have $\pr_n g\one_j=\pr_n g\two_j=g_j$
for some vector $g_j\in\Z^n$.

\begin{Thm}\label{thm:correction}
Assume that $\tB\one$ is of full rank and $b_j(t;0)$ does not vanish
for any $3\leq j\leq s$. If the following
  equation holds in $\torus\one$:
  \begin{align}\label{eq:oldRelation}
    q^{-\Hf\Lambda\one(g\one_1,g\one_2)} M\one_1*M\one_2=\sum_{3\leq j\leq s}
    c\one_j(q^{\Hf})M\one_j
  \end{align}
  for some coefficients
  $c\one_j(t)\in\tBase[x_{n+1},\ldots,x_{m\one}^\pm]$, then there
  exist unique vectors $u_j\in\N^n$ \st $g\one_j=g\one_1+g\one_2+\tB\one
  u_j$, and we have
  \begin{align}\label{eq:newRelation}
    q^{-\Hf\Lambda\two(g\two_1,g\two_2)} M\two_1*M\two_2=\sum_{3\leq j\leq s}
    c\two_j(q^{\frac{\diag}{2}})M\two_j
  \end{align}
\st the coefficients in $\tBase[x_{n+1},\ldots,x_{m\two}^\pm]$ are given by
  \begin{align}\label{eq:correctionFactor}
    c\two_j(q^{\frac{\diag}{2}})=c\one_j(q^{\frac{\diag}{2}})x^{g\two_1+g\two_2+\tB\two u_j-g\two_j}.
  \end{align}
\end{Thm}
\begin{proof}
  Expanding $\LHS$ of \eqref{eq:oldRelation}, we obtain
  \begin{align*}
    \LHS=&\sum_{v_1,v_2:v_1+v_2=v}q^{-\Hf\Lambda\one(g\one_1,g\one_2)+\Hf\Lambda\one(g\one_1+\tB\one
      v_1,g\one_2+\tB\one v_2)}\\
    &\qquad\qquad \cdot
    b_1(q^{\Hf};v_1)b_2(q^{\Hf};v_2)x^{g\one_1+g\one_2+\tB\one v}\\
=&\sum_{v_1,v_2:v_1+v_2=v}q^{\Hf(-g_1^T D\one v_2 +g_2^TD\one v_1+v_1^T BD\one v_2)}\\
    &\qquad\qquad \cdot
    b_1(q^{\Hf};v_1)b_2(q^{\Hf};v_2)x^{g\one_1+g\one_2+\tB\one v}.\\
=&\sum_{v_1,v_2:v_1+v_2=v}q^{\frac{1}{2\diag}(-g_1^T D\two v_2 +g_2^TD\two v_1+v_1^T BD\two v_2)}\\
    &\qquad\qquad \cdot
    b_1(q^{\Hf};v_1)b_2(q^{\Hf};v_2)x^{g\one_1+g\one_2+\tB\one v}.
\end{align*}
In $\eqref{eq:oldRelation}$, since $b_j(t;0)$ is nonzero, the monomial $x^{g\one_j}$ in $\RHS$ must be killed by either another monomial $x^{g\one_{j'}+\tB\one
    v_{j'}}$ in $\RHS$ or a monomial in $\LHS$. In the former case,
  repeat this argument for $x^{g\one_{j'}}$. Then, after finite steps, we
  obtain that for some $u_j\in\N^n$, $x^{g\one_j}$ must equal the
  monomial $x^{g\one_1+g\one_2+\tB\one u_j}$ in $\LHS$. Therefore, we can
  rewrite $\RHS$ as
  \begin{align*}
    \RHS=&\sum_jc\one_j(q^\Hf)\sum_{v_j}b_j(q^\Hf;v_j)x^{g\one_1+g\one_2+\tB\one
      (u_j+v_j)}.
  \end{align*}

View the both sides as usual Laurent polynomial and embed them into
the enlarged quantum torus $\torus$. We can rewrite them as
\begin{align*}
\LHS=&\sum_{v_1,v_2:v_1+v_2=v}(q^{\frac{1}{2\diag}})^{-\Lambda\two(g\two_1,g\two_2)+\Lambda\two(g\two_1+\tB\two v_1,g\two_2+\tB\two v_2)}\\
    &\qquad\qquad \cdot
    b_1(q^{\Hf};v_1)b_2(q^{\Hf};v_2)x^{g\two_1+g\two_2+\tB\two v}\\
&\qquad\qquad\cdot x^{g\one_1+g\one_2-g\two_1-g\two_2}x^{(\tB\one-\tB\two) v},\\
\RHS=&\sum_{j,v_j:u_j+v_j=v}c\one_j(q^\Hf)b_j(q^\Hf;v_j)x^{g\two_1+g\two_2+\tB\two
      (u_j+v_j)}\\
&\qquad\qquad \cdot x^{g\one_1+g\one_2-g\two_1-g\two_2}x^{(\tB\one-\tB\two)(u_j+v_j)}.
\end{align*}

In the both sides, we replace the indeterminate $q^\Hf$ by
$q^{\frac{\diag}{2}}$ and delete the factor
\begin{align*}
  x^{g\one_1+g\one_2-g\two_1-g\two_2}x^{(\tB\one-\tB\two)v}
\end{align*}
in
each monomial. Then we still have $\LHS=\RHS$, which now become
\begin{align*}
  \LHS&=q^{-\Hf \Lambda\two(g\two_1,g\two_2)}M\two_1*M\two_2\\
\RHS&=\sum_j\sum_{v_j}c\one_j(q^{\frac{\diag}{2}})b_j(q^{\frac{\diag}{2}};v_j)x^{g\two_1+g\two_2+\tB\two u_j-g\two_j}x^{g\two_j+\tB\two
  v_j}\\
&=\sum_jc\one_j(q^{\frac{\diag}{2}})x^{g\two_1+g\two_2+\tB\two u_j-g\two_j}M_j.
\end{align*}
Thus, \eqref{eq:newRelation} is proved.
\end{proof}

\begin{Rem}
  The quotient between the $q$-powers on the left of \eqref{eq:oldRelation}
  and \eqref{eq:newRelation} should be viewed as the generalization of
  the $q$-power correction factor in Failure \textbf{1}
  \eqref{eq:failureCor}. The correction factor
  $x^{g\two_1+g\two_2+\tB\two u_j-g\two_j}$ should be viewed as the
  generalization of the factor in Failure \textbf{2}
  \eqref{eq:failureContr}. The difference $g\two_j-g\one_j$ should be
  viewed as the generalization of the difference of the two sides of \eqref{eq:failureInd}.
\end{Rem}

\subsection{Structure constants}


\begin{Thm}\label{thm:quantumClusterAlgebraChange}
For $i=1$, $2$, assume that the pair $(\tB\typeI,\Lambda\typeI)$, $i=1,2$, in
Theorem \ref{thm:correction} is either compatible or the
matrix $\Lambda\typeI$ is zero. Denote the
corresponding quantum or classical cluster algebras by $\qClAlg\typeI$.



For $i=1$, $2$, let $\basis\typeI$ be a $\Z[q^{\pm\Hf}][x_{n+1}^\pm,\ldots,x_{m\typeI}^\pm]$-basis of $\qClAlg\typeI$ \st its elements take the form $M\typeI_j$. If
the structure constants of $\basis\one$ are described by
\eqref{eq:oldRelation}, then the structure constants of $\basis\two$ are described by the
corresponding equation \eqref{eq:newRelation}.
\end{Thm}
\begin{proof}

The statement is obtained
  by taking \eqref{eq:oldRelation} to be the multiplication of basis elements.
\end{proof}

\begin{Rem}
By this theorem, in order to study the structure constants of the bases, it suffices to choose a special coefficient pattern with
special quantizations, \eg the
principal coefficients with the canonical quantization.


\end{Rem}

\subsection{Acyclic case}
\label{sec:classical}

Now we can apply the correction techniques to the results of Section
\ref{sec:bases}. We simplify the proofs of previous results on bases
of acyclic quantum cluster algebras, and we also present results about the
bases of acyclic classical cluster algebras.

Let the compatible pairs $(\tB\typeI,\Lambda\typeI)$, $i=1,2$,
be given as in Theorem \ref{thm:quantumClusterAlgebraChange}.
Further assume that $B=B_Q$ for some acyclic
quiver, and choose $(\tB\one,\Lambda\one)$ to be
$(\ztB,\zLambda)$. Then we can choose the $\qBase$-linearly independent
subset\footnote{The subset $\basis\one$ is a basis of the subalgebra
of $\qClAlg\one$ generated by the quantum cluster variables and the
frozen variables $x_{n+1},\ldots,x_{2n}$.} $\basis\one$ of the
elements of $\qClAlg\one$
to be the set $\genClBasis$, the set $\pbwClBasis$, or the set
$\canClBasis$. Notice that the elements of $\basis\one$ take the form $M\one_j$ of
Theorem \ref{thm:correction}. Define $\basis\two$ to be
the subset $\clGenBasis$, $\clPbwBasis$, $\clCanBasis$ in $\torus\two$
respectively if
$\Lambda\two$ is nonzero, or the corresponding classical limit if
$\Lambda\two=0$. Then the elements of $\basis\two$ take the form $M\two_j$.

\begin{Thm}\label{thm:acyclicBasis}
If $\tB\two$ is of full rank, then $\basis\two$ is a $\qBaseCoeff$-basis of
$\qClAlg\two$. Furthermore, if $\basis\two$ is $\clGenBasis$ or
$\clCanBasis$, then the structure constants of $\basis\two$ can be
deduced from \eqref{eq:newRelation} which corresponds to the structure
constants equation \eqref{eq:oldRelation}.
\end{Thm}
\begin{proof}
By the proof of \ref{prop:differentIndex}, the leading terms (terms with
$v=0$) of the
elements in $\basis\two$ are different. Because $\tB\two$ is
of full rank, in each $M\two_j$, the leading term cannot be killed by the
other terms. Therefore, $\basis\two$ is linearly independent. Because $\basis\one$ is contained in $\qClAlg\one$, by Theorem
\ref{thm:correction}, $\basis\two$ is contained in
$\qClAlg\two$. 

Assume $\basis\two$ is either $\clGenBasis$ or
$\clCanBasis$. Then it contains all the quantum cluster monomials by section
\ref{sec:pseudoModules}. Moreover, it generates a $\qBaseCoeff$-subalgebra of
$\qClAlg\two$ by Theorem \ref{thm:correction} and
\eqref{eq:coeffFactorization} or \eqref{eq:canFactorization}. The statement concerning its structure constants also follows from Theorem
\ref{thm:correction}.

Using the upper unitriangular basis transition matrices
studied before, we deduce that $\clPbwBasis$ is also a basis.
\end{proof}

\begin{Rem}
If we take $\Lambda\two=0$, then $\qClAlg\two$ is a classical cluster
algebra, which we denote by $\clAlg\two$. 

1) If we take $\basis\one$ to be the dual canonical basis, then we
obtain that $\basis\two$ is the dual canonical basis of $\clAlg\two$
with positive structure constants. All cluster monomials are contained
in $\basis\two$.

2) If we take $\basis\one$ to be the generic basis
and the dual PBW basis, then this theorem generalizes the results of the
dual semicanonical basis and the dual PBW basis in
\cite[16.1]{GeissLeclercSchroeer10} to the integral form and general coefficient type.

Recently, another proof of the generic basis for $\clAlg\two$ is
obtained in \cite{GeissLeclercSchroeer10b} (see also \cite{Plamondon10c}),
which is based on cluster categories and a combinatorial
result of \cite{BerensteinFominZelevinsky05}.
\end{Rem}

\section*{Acknowledgments}\label{sec:ack}
  The author would like to express his sincere thanks to his thesis
  advisor Bernhard Keller for the encouragement and discussions. He
  also thanks Yoshiyuki Kimura for many useful discussions, in
  particular for the explanations of his previous work on the factorization of dual
  canonical basis and of Nakajima's quiver varieties. He is grateful
  to the referee for many helpful suggestions for improving the
  readability of this article.




\def\cprime{$'$}
\providecommand{\bysame}{\leavevmode\hbox to3em{\hrulefill}\thinspace}
\providecommand{\MR}{\relax\ifhmode\unskip\space\fi MR }
\providecommand{\MRhref}[2]{%
  \href{http://www.ams.org/mathscinet-getitem?mr=#1}{#2}
}
\providecommand{\href}[2]{#2}


\end{document}